\documentclass[12pt]{amsart}
\usepackage[psamsfonts]{amssymb}
\def\R{\mathbb R}
\def\Z{\mathbb Z}
\def\C{\mathbb C}
\newtheorem{thm}{Theorem}[section]
\newtheorem{prop}[thm]{Proposition}
\newtheorem{cor}[thm]{Corollary}
\newtheorem{lemma}[thm]{Lemma}
\theoremstyle{remark}
\newtheorem{remark}[thm]{Remark}
\numberwithin{equation}{section}
\begin{document}
%
\title[The cut loci on ellipsoids]{The cut 
loci on ellipsoids and certain Liouville manifolds}
\author{Jin-ichi Itoh}
\address{Department of Mathematics, Faculty of Education,
Kumamoto University, Kumamoto 860-8555, Japan.}
\email{j-itoh@gpo.kumamoto-u.ac.jp}
\author{Kazuyoshi Kiyohara}
\address{Department of Mathematics, Faculty of Science, Okayama University, Okayama 700-8530, Japan.}
\email{kiyohara@math.okayama-u.ac.jp}
\begin{abstract}
We show that some riemannian manifolds
diffeomorphic to the sphere have the property that
the cut loci of general
points are smoothly embedded closed disks of codimension one.
Ellipsoids with distinct axes are typical examples of
such manifolds.
\end{abstract}
\maketitle
\section{Introduction}
On a complete riemannian manifold, any geodesic $\gamma(t)$
starting at a point $\gamma(0)=p$ has the property that
any segment $\{\gamma(t)\mid 0\le t\le T\}$ is minimal, 
i.e., the length of the segment is equal to the distance
between the points $p$ and $\gamma(T)$, if $T>0$ is
small. 
If the supremum $t_0$ of the set of such $T$ is
finite, then the point $\gamma(t_0)$ is called the {\it cut 
point} of $p$ along the geodesic $\gamma(t)$ $(t\ge 0)$.
The {\it cut locus} of the point $p$ is then defined as the set of all cut
points of $p$ along the geodesics starting at $p$.
For the general properties of cut loci, we refer to
\cite{Kl}, \cite{Sa2}.

The study of cut locus was started at 1905 by H. Poincar\'{e} \cite{P} 
in the case of convex surfaces, and there are several classical results, 
for example, \cite{M}, \cite{We}, \cite{Wh}. 
From its definition, the cut locus of a point $p$ on a compact manifold $M$ 
is homotopically equivalent to $M - \{p\}$,
but it can be very complicated, see \cite{GS}, \cite{I}. 
The structure of cut locus was studied in 
connection with 
the singularity theory, see \cite{Bu}, \cite{Bu1}, \cite{Thom}. 
Recently, a property of cut locus was used to solve Ambrose's problem on surfaces 
\cite{Heb}, \cite{I}, and it was proved that the distance function to the cut locus has 
Lipschitz continuity \cite{IT}, \cite{LN}. 
Other applications of cut locus are found in \cite{DO}, \cite{LN} also. 

It is well known that the cut locus of any point on the sphere of constant
curvature consists of a single point, and
it is also known that this property 
characterizes the sphere of constant curvature (an 
affirmatively solved case of the
Blaschke conjecture, see \cite{B}).
However, in most cases, to determine cut loci are quite difficult problems.
There are only a few cases where the cut loci are well
understood; for example, analytic surfaces \cite{M},
symmetric spaces and some homogeneous spaces \cite{H}, 
\cite{Sa-1}, \cite{Sa0}, \cite{Sa1}, \cite{Ta},
certain surfaces of revolution \cite{GMST}, \cite{ST2}, \cite{T1}, \cite{T2}, 
Alexandrov surfaces \cite{ShiT}, 
tri-axial ellipsoids and some Liouville
surfaces \cite{IK1}, \cite{IK2}, \cite{ST1}
(\cite{ST1} is an experimental work). 
Especially in higher dimensional case 
there are not many results without symmetric spaces and some singular spaces \cite{IV}, 
even if using computational approximations.  

In the earlier paper \cite{IK1}, we proved that the cut
locus of a non-umbilic point on a tri-axial ellipsoid is a
segment of the curvature line containing the antipodal
point, inspired by an experimental work \cite{IS}. Also, 
we gave the complete proof of Jacobi's last geometric statement (\cite{J}, \cite{JW}, see also
\cite{Si}, which contains historical remarks). 
Furthermore, we have seen in \cite{IK2} 
that there are many
surfaces possessing such simple cut loci.
Surfaces we considered in \cite{IK2} are
so-called Liouville  surfaces, i.e., surfaces
whose geodesic flows possess
first integrals 
which are fiberwise quadratic forms.
In such cases the geodesic equations are
explicitly solved by quadratures.
But, to determine cut loci we needed some 
additional conditions, which is satisfied in
the case of ellipsoid.

In the present paper, we shall give a higher dimensional
version of the above-mentioned results. We shall consider
cut loci of points on certain Liouville manifolds
diffeomorphic to $n$-sphere, and prove that the cut locus 
of any point is a smoothly embedded, closed 
$(n-1)$-disk, if the point does not
belong to a certain submanifold of codimension two. We shall 
also prove that the cut locus of a point on that submanifold 
is a closed $(n-2)$-disk.
The $n$-dimensional ellipsoids with $n+1$ distinct axes
will be shown to possess such properties. 
Here, ``Liouville manifold'' is a
higher dimensional version of Liouville surface, which
we shall explain in the next section.

Now, taking the ellipsoid $M: \sum_{i=0}^n  u_i^2/a_i=1$ ($
0<a_n<\dots<a_0$) as an example, let us illustrate our results
in detail. Let $N_k$ and $J_k$ be the submanifolds of $M$ defined by
\begin{gather*}
N_k=\{u=(u_0,\dots,u_n)\in M\ |\ u_k=0\ \}\qquad(0\le k\le n)\\
J_k=\{u\in M\ |\ u_k=0,\quad \sum_{i\ne k}\frac{u_i^2}{a_i-a_k}=1\ \}\qquad (1\le k\le n-1)
\end{gather*}
Then: $N_k$ is totally geodesic, codimension 1; $J_k
\subset N_k$, $J_k$ is diffeomorphic to $S^{k-1}\times
S^{n-k-1}$; $\bigcup_kJ_k$ is the set of points where
some principal curvature with respect to the inclusion
$M\subset \R^{n+1}$ has multiplicity $\ge 2$;
denoting by $(\lambda_1,\dots,\lambda_n)$ the 
elliptic coordinate system on $M$ such that
$a_k\le \lambda_k\le
a_{k-1}$ (see below), we have
\begin{equation*}
N_k=\{\lambda_k=a_k\quad\text{or}\quad\lambda_{k+1}=a_k\ \},
\qquad J_k=\{\lambda_k=\lambda_{k+1}=a_k\ \} .
\end{equation*}
Let us denote by $C(p)$ the cut locus of a point $p\in M$.
Let $(\lambda_1^0,\dots,\lambda_n^0)$ be the elliptic 
coordinates of $p$.
Then:
\begin{itemize}
\item[(1)] If $p\not\in J_{n-1}$, then $C(p)$ is an 
$(n-1)$-dimensional closed disk which is contained in
a submanifold (possibly with boundary) defined by 
$\lambda_n=\lambda_n^0$. Also,
$C(p)$ contains the antipodal point of $p$ in its interior.
For each interior point $q$ of $C(p)$ there are exactly two
minimal geodesics joining $p$ and $q$; the tangent vectors
of those geodesics at $p$ are symmetric with respect to the
hyperplane $d\lambda_n=0$. For each boundary point $q$ of 
$C(p)$, there is a unique minimal geodesic from $p$ to $q$, 
along which $q$ is the first conjugate point of $p$ with
multiplicity one.
\item[(2)] If $p\in J_{n-1}$, then $C(p)$ is an 
$(n-2)$-dimensional closed disk contained in $J_{n-1}$. It
is identical with the cut locus of $p$ in the 
$(n-1)$-dimensional ellipsoid $N_{n-1}$. For each interior 
point $q$ of $C(p)$ there is an $S^1$-family of minimal 
geodesics joining $p$ and $q$; the tangent vectors
of those geodesics at $p$ form a cone whose orthogonal 
projection to $T_pJ_{n-1}$ is one-dimensional. For each boundary point $q$ of 
$C(p)$, there is a unique minimal geodesic from $p$ to $q$, 
and along it $q$ is the first conjugate point of $p$; but the
multiplicity is two in this case.
\end{itemize}
Here, the elliptic coordinate system $(\lambda_1,\dots,\lambda_n)$ on $M$ $(\lambda_n\le \dots\le\lambda_1)$ is defined by the following identity in $\lambda$:
\begin{equation*}
\sum_{i=0}^n\frac{u_i^2}{a_i-\lambda}-1=\frac{\lambda\prod_{k=1}^n(\lambda_k-\lambda)}{\prod_i (a_i-\lambda)} .
\end{equation*}
For a fixed $u\in M$, $\lambda_k$ are determined by $n$ 
``confocal quadrics'' passing through $u$. 
From $\lambda_k$'s, $u_i$ are explicitly described as:
\begin{equation*}
u_i^2=\frac{a_i\prod_{k=1}^n(\lambda_k-a_i)}{\prod_{j\ne i} (a_j-a_i)} .
\end{equation*}

The organization of the paper is as follows.
In \S2 we shall briefly explain Liouville manifolds in the form
what we need. 
In \S 3 we shall illustrate how to solve geodesic 
equations on a Liouville manifold. Since the geodesic flow is
completely integrable in this case, solutions are given by
integrating a system of closed $1$-forms. In this particular 
case, a natural coordinate system provides ``separation
of variables''. This coordinate system is analogous to the
elliptic coordinate system on ellipsoids. In \S4 we shall give 
an assumption under which the results on cut loci are
obtained. Some useful inequalities are proved there.

In \S5 basic properties of Jacobi fields and their
zeros are
investigated, which are crucial in the arguments
of the following sections. In \S6 we define a value $t_0(\eta)$ to each unit covector $\eta$, which will indicate
the cut point of the geodesic with initial covector $\eta$. 
Then, we prove some preliminary
facts on the behavior of geodesics starting at
a fixed point. The main theorem,
Theorem \ref{thm:cut}, will be stated in \S7 and 
proved in \S\S 7-9.

In the forthcoming paper, we shall clarify the 
structures of conjugate loci of general points on
certain Liouville manifolds, which will be a higher
dimensional version of ``the last geometric statement
of Jacobi'' explained in \cite{IK1}, \cite{Si}.

\subsection*{Preliminary remarks and notations}
We shall consider geodesic equations in the hamiltonian formulation.
Let $M$ be a riemannian manifold and $g$ its riemannian 
metric.
By $\flat:TM\to T^*M$ we denote the bundle isomorphism
determined by $g$ (Legendre transformation). We also use the
symbol $\sharp=\flat^{-1}$.
The canonical 1-form on $T^*M$ is denoted by $\alpha$.
For a canonical coordinate system $(x,\xi)$ on $T^*M$ ($
x$ being a coordinate system on $M$), $\alpha$ is expressed as
$\sum_i\xi_idx_i$. Then the 2-form $d\alpha$ 
represents the standard
symplectic structure on $T^*M$.

Let $E$ be the function on $T^*M$ defined by 
\begin{equation*}
E(\lambda)=\frac12g(\sharp(\lambda),\sharp(\lambda))=
\frac12\sum_{i,j}g^{ij}(x)\xi_i\xi_j
\end{equation*}
We call it the (kinetic) energy function of $M$.
For a function $F, H$ on $T^*M$, we define a vector field 
$X_F$ and the Poisson bracket $\{F,H\}$ by
\begin{equation*}
X_F=\sum_i\left(\frac{\partial F}{\partial \xi_i}\frac{\partial}
{\partial x_i}-\frac{\partial F}{\partial x_i}\frac{\partial}
{\partial \xi_i}\right)\, ,\qquad
\{F,H\}=X_FH\,.
\end{equation*}
Then $X_E$ generates the geodesic flow, i.e., the projection of each
integral curve of $X_E$ to $M$ is a geodesic of the riemannian manifold
$M$.

\section{Liouville manifolds}
By definition, Liouville manifold $(M,\mathcal F)$ is a pair 
of an $n$-dimensional riemannian manifold $M$ and 
an $n$-dimensional vector space
$\mathcal F$ of functions on $T^*M$ such that i) each $F\in \mathcal F$ is
fiberwise a quadratic polynomial; ii) those quadratic forms are
simultaneously normalizable on each fiber; iii) $\mathcal F$ is commutative with respect to the Poisson bracket; and, iv)
$\mathcal F$ contains the hamiltonian of the geodesic flow.
For the general theory of Liouville manifolds, we refer to \cite{Ki2}.
In this paper we only need a subclass of ``compact Liouville
manifolds of rank one and type (A)'', described in \cite{Ki2}.
So, in this section,
we shall briefly explain about it.

Each Liouville manifold treated here is constructed
from $n+1$ constants $a_0>\cdots > a_{n}>0$ and a positive $C^\infty$
function $A(\lambda)$ on the closed interval $a_n\le \lambda\le a_0$. 
Let $\alpha_1,\dots,\alpha_n$ be positive numbers defined by
\begin{equation*}
\alpha_i=2\int_{a_i}^{a_{i-1}}\frac{A(\lambda)\ d\lambda}
{\sqrt{(-1)^i\prod_{j=0}^n (\lambda-a_j)}}\qquad
(i=1,\dots,n)
\end{equation*}
Define the function $f_i$ on the circle $\R/\alpha_i\Z=
\{x_i\}$ $(1\le i\le n)$
by the conditions:
\begin{gather}\label{eq:fdiff}
\left(\frac{df_i}{dx_i}\right)^2=\frac{(-1)^i4\prod_{j=0}
^n (f_i-a_j)}{A(f_i)^2}\\
f_i(0)=a_i,\ f_i(\frac{\alpha_i}4)=a_{i-1},\quad
f_i(-x_i)=f_i(x_i)=f_i(\frac{\alpha_i}2-x_i)\ .
\end{gather}
Then the range of $f_i$ is $[a_i,a_{i-1}]$.  

Put
\begin{equation*}
R=\prod_{i=1}^n(\R/\alpha_i\Z)\ .
\end{equation*}
Let $\tau_i$ $(1\le i\le n-1)$ be the involutions
on the torus $R$ defined by
\begin{equation*}
\tau_i(x_1,\dots,x_n)=(x_1,\dots,x_{i-1},-x_i,
\frac{\alpha_{i+1} }2-x_{i+1},x_{i+2},\dots,x_n) \ ,
\end{equation*}
and let $G$ $(\simeq (\Z/2\Z)^{n-1})$ be the group of transformations generated by $\tau_1$, $\dots$,
$\tau_{n-1}$. Then it turns out that the quotient space
$M=R/G$ is homeomorphic to the $n$-sphere. Moreover,
let $p\in R$ be a ramification point of the branched covering
$R\to R/G$. Suppose $p$ is fixed by $\tau_{i_1}$,
$\dots$, $\tau_{i_k}$, and is not fixed by other $\tau_j$'s. Taking a suitable coordinate
system $(y_1,\dots,y_n)$ obtained from $(x)$ by exchanges $(x_i\leftrightarrow x_j)$ and translations $(x_i\to x_i+c)$, 
it may be supposed that $p$ is represented by $y=0$ and 
$\tau_{i_l}$ is given by
\begin{equation*}
(y_1,\dots,y_n)\mapsto (y_1,\dots,y_{2l-2},-y_{2l-1},
-y_{2l},y_{2l+1},\dots,y_n)\ .
\end{equation*}
Then we can define a differentiable structure on $M$ so
that 
\begin{equation*}
(y_1^2-y_2^2,2y_1y_2,\dots,y_{2k-1}^2-y_{2k}^2,
2y_{2k-1}y_{2k},y_{2k+1},\dots,y_n)
\end{equation*}
is a smooth coordinate system around the image of $p$.
With this $M$ is diffeomorphic to
the standard $n$-sphere.
One can prove those facts by comparing the branched
covering $R\to R/G$ with the standard case; see
\cite[p.73]{Ki2}.

Now, put
\begin{equation*}
b_{ij}(x_i)=
\begin{cases}
(-1)^i\prod_{\substack{1\le k\le n-1\\ k\ne j}}(f_i(x_i)-a_k)\quad (1\le j\le n-1)\\
(-1)^{i+1}\prod_{k=1}^{n-1}(f_i(x_i)-a_k)\qquad (j=n)
\end{cases}\ ,
\end{equation*}
and define functions $F_1$, $\dots$, $F_n=2E$ on the cotangent
bundle by
\begin{equation}\label{integrals}
\sum_{j=1}^n b_{ij}(x_i)F_j=\xi_i^2\ ,
\end{equation}
where $\xi_i$ are the fiber coordinates with respect to
the base coordinates $(x_1,\dots,x_n)$. Although there
are points on $T^*R$ where $F_i$ are not well-defined,
it turns out that $F_i$ represent well-defined smooth
functions on $T^*M$. Computing the inverse matrix of $(b_{ij})$
explicitly, we have
\begin{align*}
2E=&\sum_i\frac{(-1)^{n-i} \xi^2_i}{\prod_{l\ne i}
(f_l-f_i)}\\
F_j=&\frac1{\prod_{\substack{1\le k\le n-1\\k\ne j}}(a_k-a_j)}
\sum_i\frac{(-1)^{n-i}\prod_{l\ne i}(f_l-a_j)}
{\prod_{l\ne i}(f_l-f_i)}\ \xi^2_i\qquad
(j\le n-1)\ .
\end{align*}

One can also see that $E$, restricted to each
cotangent space of $M$,  is a positive definite quadratic
form.  Therefore
\begin{equation}\label{eq:metric}
g=\sum_i(-1)^{n-i}\left(\prod_{l\ne i}(f_l-f_i)\right)\,dx_i^2
\end{equation}
is a well-defined riemannian metric on $M$, and $E$
is the hamiltonian of the associated geodesic flow.
We call $E$ the energy function of the riemannian manifold 
$M$. From the formula (\ref{integrals}) one can easily see 
that
\begin{equation*}
\{F_i,F_j\}=0\qquad (1\le i,j\le n)\ ,
\end{equation*}
where $\{,\}$ denotes the Poisson bracket (see 
\cite[Prop.\,1.2.3]{Ki2}). In particular, the geodesic flow
is completely integrable in the sense of hamiltonian
mechanics.

As examples, if $A(\lambda)$ is a constant function, then
$M$ is the sphere of constant curvature. This case is
explained in detail in \cite[pp.71--74]{Ki2}.  If $A(\lambda)
=\sqrt{\lambda}$, then $M$ is isometric to the ellipsoid
$\sum_{i=0}^n{u_i^2}/{a_i}=1$. In this case, the system 
of functions
$(f_1(x_1),\dots,f_n(x_n))$ 
is nothing but the elliptic coordinate system (see Introduction), i.e., $f_i(x_i)=\lambda_i$.
One can easily check that the induced metric $\sum_i du_i^2$
coincides with the formula (\ref{eq:metric}) when $f_i$
satisfy the equations (\ref{eq:fdiff}) and
$A(\lambda)=\sqrt{\lambda}$.

Finally, let us define certain submanifolds of $M$ which are 
analogous to those for the ellipsoid stated in 
Introduction:  Put
\begin{gather*}
N_k=\{x\in M\ |\ f_k(x_k)=a_k \quad\text{or}\quad
f_{k+1}(x_{k+1})=a_k\}\quad(0\le k\le n) ,\\
J_k=\{x\in M\ |\ f_k(x_k)=
f_{k+1}(x_{k+1})=a_k\}\quad (1\le k\le n-1) .
\end{gather*}
Then we have, putting $(F_k)_p=F_k\vert_{T_p^*M}$,
\begin{lemma}\label{lemma:subm}
\begin{itemize}
\item[(1)] $J_k=\{p\in M\ |\ (F_k)_p=0\}$.
\item[(2)] $N_k=\{p\in M\ |\ \text{\rm rank }(F_k)_p\le 1\}$
\quad $(1\le k\le n-1)$.
\item[(3)] $\bigcup_k J_k$ is identical with the branch
locus of the covering $R\to M=R/G$.
\item[(4)] $N_k$ is a totally geodesic submanifold of codimension one \ $(0\le k\le n)$.
\item[(5)] $J_k\subset N_k$, $J_k$ is diffeomorphic to
$S^{k-1}\times S^{n-k-1}$.
\end{itemize}
\end{lemma}
\begin{proof} For (1) and (2), see \cite[pp.52--56]{Ki2}.
(3) is obvious.
(4) follows from the fact that $N_k$ is the fixed point-set
of the involutive isometry $(x_1,\dots,x_n)\mapsto (x_1,\dots
,-x_k,\dots,x_n)$. (5) is easily seen by comparing the branched covering with the standard one, \cite[p.73]{Ki2}.
\end{proof}
\section{Geodesic equations}
The geodesic equations are generally written as
\begin{equation*}
\frac{dx_i}{dt}=\frac{\partial E}{\partial \xi_i},
\quad \frac{d\xi_i}{dt}=-\frac{\partial E}
{\partial x_i}.
\end{equation*}
But, since our geodesic flow is completely integrable, it is better to consider the equation
of geodesics with $F_j=c_j$ $(1\le j\le n-1)$ 
and $2E=1$. If $c=(c_1,\dots,c_{n-1},1)$ is a
regular value of the map
\begin{equation*}
\boldsymbol F=(F_1,\dots, F_{n-1},2E): T^*M\to \R^n\ ,
\end{equation*}
then its inverse image is a disjoint union of tori, and
the vector fields $X_{F_j}$, $X_E$ on it are
mutually commutative and linearly independent
everywhere. Here $X_f$ denotes the hamiltonian vector field
determined by a function $f$; 
\begin{equation*}
X_f=\sum_i\left(\frac{\partial f}{\partial \xi_i}
\frac{\partial }{\partial x_i}-
\frac{\partial f}{\partial x_i}
\frac{\partial }{\partial \xi_i}\right)\ .
\end{equation*}

Let $\omega_j$ $(1\le j\le n)$ be
the dual $1$-forms of $\{\pi_*X_{F_j}\}$, where
$\pi:T^*M\to M$ is the bundle projection. Then, by 
(\ref{integrals}) we have
\begin{equation*}
\omega_l=\sum_i\frac{b_{il}}{2\xi_i}\ dx_i\qquad
(1\le l\le n).
\end{equation*}
They are closed $1$-forms, and the geodesic
orbits are determined by
\begin{equation}\label{eq:geod0}
\omega_l=0 \qquad (1\le l\le n-1),
\end{equation}
and the length parameter $t$ on an orbit
is given by
\begin{equation}\label{eq:length}
dt=2\omega_n.
\end{equation}

Putting
\begin{equation*}
\Theta(\lambda)=\sum_{j=1}^{n-1}\left(
\prod_{\substack{1\le k\le n-1\\ k\ne j}}
(\lambda-a_k)\right)\,c_j-\prod_{k=1}^{n-1}(\lambda-a_k)\ ,
\end{equation*}
we have from (\ref{integrals})
\begin{equation*}
\xi_i=\epsilon_i\sqrt{\sum_j b_{ij}(x_i)c_j}
=\epsilon_i\sqrt{(-1)^i\Theta(f_i(x_i))}\qquad (1\le i\le n),
\end{equation*}
where $\epsilon_i=\text{sgn}\,\xi_i=
\text{sgn}\,\left(\frac{dx_i}{dt}\right)=\pm 1$.
If a covector $(x,\xi)$ with $F_1=c_1,\dots,
F_{n-1}=c_{n-1}$, $2E=1$ satisfies
$\xi_i\ne 0$
for any $1\le i\le n$, then we have
\begin{equation*}
(-1)^i \Theta(f_i(x_i))>0.
\end{equation*}
Therefore for such $c_1,\dots,c_{n-1}$, the equation
$\Theta(\lambda)=0$ has $n-1$ distinct real roots
$b_1>b_2>\dots>b_{n-1}$, and they satisfy
\begin{equation*}
f_1(x_1)>b_1>f_2(x_2)>b_2>\dots>f_{n-1}(x_{n-1})>b_{n-1}>f_n(x_n).
\end{equation*}
Thus we have the identity
\begin{equation*}
\Theta(\lambda)=-\prod_{l=1}^{n-1} (\lambda-b_l),
\end{equation*}
and $c_j$ are expressed by $b_l$'s as
\begin{equation}\label{b-c}
c_j=\frac{-\prod_{l=1}^{n-1}(a_j-b_l)}
{\prod_{\substack{1\le k\le n-1\\ k\ne j}}
(a_j-a_k)}
\qquad (1\le j\le n-1)\ .
\end{equation}

Conversely, let $b_1,\dots,b_{n-1}$ be any real numbers satisfying
\begin{equation}\label{brange}
a_{i+1}\le b_i\le a_{i-1}\ ,\qquad b_{i+1}\le b_i
\end{equation}
for any $i$, and define $c_1,\dots,c_{n-1}$ by (\ref{b-c}).
Then there is a covector $(x,\xi)$ with $F_1=c_1$, $\dots$, 
$F_{n-1}=c_{n-1}$, $2E=1$. It can be verified that if
$b_1$, $\dots$, $b_{n-1}$ satisfy 
\begin{equation}\label{b-cond}
a_{i+1}< b_i< a_{i-1}\ ,\quad b_i\ne a_i,\quad 
b_{i+1}< b_i\qquad \text{for any }i
\end{equation}
then the corresponding $c=(c_1,\dots,c_{n-1},1)$ is a regular
value of $\boldsymbol F$.

To describe the behavior of the geodesics it is more convenient to
use the values $(b_1,\dots,b_{n-1})$ rather than using
$(c_1,\dots,c_{n-1})$ directly. So, we shall mainly use 
$(b_1,\dots,b_{n-1})$ as the values of first integrals which determine
the Lagrange tori $\boldsymbol F^{-1}(c)$.
Also, we shall denote by $H_1,\dots,H_{n-1}$ the functions on
the unit cotangent bundle $U^*M$ whose values are $b_1,\dots,b_{n-1}$. Namely, $H_i$'s are determined by
\begin{gather*}
F_j(\mu)=\frac{-\prod_{l=1}^{n-1}(a_j-H_l(\mu))}
{\prod_{\substack{1\le k\le n-1\\ k\ne j}}
(a_j-a_k)}
\qquad (1\le j\le n-1),\\
H_1(\mu)\ge\dots\ge H_{n-1}(\mu),\qquad \mu\in U^*M.
\end{gather*}
The range of $H_i$ are given by (\ref{brange}).

Now, put
\begin{align*}
&a_i^+=\max\{a_i,b_i\}\quad (1\le i\le n-1),
\quad a_n^+=a_n\\
&a_i^-=\min\{a_i,b_i\}\quad (1\le i\le n-1),
\quad a_0^-=a_0\ .
\end{align*}
If $b_1,\dots,b_{n-1}$ satisfy the condition (\ref{b-cond}), then
the $\pi$-image of a connected component of
$\boldsymbol F^{-1}(c)$ (a Lagrange torus) is of 
the form
\begin{equation*}
L_1\times\dots \times L_n\subset M\ ,
\end{equation*}
where each $L_i$ is a connected component of the inverse
image of $[a_i^+,a_{i-1}^-]$ by the map
\begin{equation*}
f_i:\R/\alpha_i\Z\to [a_i,a_{i-1}]\ .
\end{equation*}
(Observe that the ``generalized band'' $L_1\times\dots\times 
L_n\subset R$ is injectively mapped to $M$
by the branched covering $R\to M$.)

Along a geodesic $(x_1(t),\dots,x_n(t))$, the coordinate function $x_i(t)$ oscillates on $L_i$ if $L_i$ is an interval,
or $x_i(t)$ moves monotonously if $L_i$ is the whole circle.
Also, the function $f_i(x_i(t))$ oscillates on the interval
$[a_i^+,a_{i-1}^-]$

After all, the equations of geodesic orbits
\begin{equation*}
\omega_l=0\qquad (1\le l\le n-1)
\end{equation*}
are described as
\begin{equation*}
\sum_{i=1}^n\frac{\epsilon_i(-1)^i\prod_{\substack{1\le k\le n-1\\ k\ne l}}(f_i(x_i)-a_k)\ dx_i}
{\sqrt{(-1)^{i-1}\prod_{k=1}^{n-1}(f_i(x_i)-b_k)}}
=0\qquad (1\le l\le n-1)\ .
\end{equation*}
Note that this system of equations is equivalent to
\begin{equation*}
\sum_{i=1}^n\frac{\epsilon_i(-1)^iG(f_i)\ dx_i}
{\sqrt{(-1)^{i-1}\prod_{k=1}^{n-1}(f_i-b_k)}}
=0
\end{equation*}
for any polynomial $G(\lambda)$ of degree $\le n-2$.
Since
\begin{equation*}
\left(\frac{df_i}{dx_i}\right)^2=\frac{(-1)^i4
\prod_{k=0}^n(f_i-a_k)}{A(f_i)^2},
\end{equation*}
those equations are also described as
\begin{equation}\label{eq:geod}
\sum_{i=1}^n\frac{\epsilon'_i(-1)^iG(f_i)A(f_i)\ df_i}
{\sqrt{-\prod_{k=1}^{n-1}(f_i-b_k)
\cdot\prod_{k=0}^n(f_i-a_k)}}
=0,
\end{equation}
where $\epsilon_i'=\text{sgn of }df_i(x_i(t))/dt$.

By (\ref{eq:geod}) we have
\begin{equation*}
\sum_{i=1}^n\int_s^{t}\frac{(-1)^iG(f_i)A(f_i)}
{\sqrt{-\prod_{k=1}^{n-1}(f_i-b_k)
\cdot\prod_{k=0}^n(f_i-a_k)}}\ 
\left|\frac{df_i(x_i(t))}{dt}\right|
\ dt=0
\end{equation*}
for any polynomial $G(\lambda)$ of degree $\le n-2$ and for
a fixed $s\in \R$.
By using the variables $\sigma_i$ defined by
\begin{equation*}
\sigma_{i}(t)=\int_0^t\left|\frac{df_i(x_i(t))}{dt}\right|
\ dt\ ,
\end{equation*}
this formula
is rewritten as
\begin{equation}\label{eq:geodsigma1}
\sum_{i=1}^n\int_{\sigma_i(s)}^{\sigma_i(t)}\frac{(-1)^iG(f_i)A(f_i)}
{\sqrt{-\prod_{k=1}^{n-1}(f_i-b_k)
\cdot\prod_{k=0}^n(f_i-a_k)}}\ 
\ d\sigma_i=0\ .
\end{equation}
Here, $f_i$ is regarded as a function of $\sigma_i$, i.e.,
putting $\phi_i(t)
=a_i+|t|$ for $|t|\le a_{i-1}-a_i$ and extending it to $\R$
as a periodic function with the period $2(a_{i-1}-a_i)$, 
we have
\begin{equation*}
f_i=
\phi_i(\sigma_i+\epsilon_i(f_i(x_i(0))-a_i))\ ,
\end{equation*}
where $\epsilon_i=\pm 1$ is the sign of $df_i(x_i(t))/dt$ at 
$t=0$.
Also, integrating $dt=\sum_i(b_{in}/\xi_i)dx_i$, we have 
\begin{equation}\label{eq:geodlength}
\sum_{i=1}^n\int_{\sigma_i(s)}^{\sigma_i(t)}
\frac{(-1)^i\,\tilde G(f_i)A(f_i)}
{2\sqrt{-\prod_{k=1}^{n-1}(f_i-b_k)
\cdot\prod_{k=0}^n(f_i-a_k)}}\ 
d\sigma_i=t-s\ ,
\end{equation}
where $\tilde G(\lambda)$ is any monic polynomial in 
$\lambda$ of degree $n-1$.
\section{A monotonicity condition for $A(\lambda)$}
We put the following conditions on the function
$A(\lambda)$:
\begin{equation}\label{cond2}
(-1)^{k-1}A^{(k)}(\lambda)>0\quad \text{on }
[a_n,a_0]\qquad (1\le k\le n-1)
\end{equation}
for $n\ge 3$, where $A^{(k)}$ denotes the $k$-th derivative of $A$.
For the case $n=\dim M=2$, we need (\ref{cond2}) for $1\le k\le 2$, 
as described in our earlier paper \cite{IK2}.
A typical example satisfying the condition 
(\ref{cond2}) is the ellipsoid,
in which case $A(\lambda)=\sqrt{\lambda}$.
Since the condition (\ref{cond2}) is $C^{n-1}$-open,
there are surely many $A(\lambda)$ satisfying it.

In the rest of this section, we shall prove some 
inequalities which are obtained under the condition
(\ref{cond2}).
Put
\begin{equation*}
G_l(\lambda)=\prod_{\substack{1\le k\le n-1\\
k\ne l}}(\lambda-b_k)\qquad (1\le l\le n-1)\ .
\end{equation*}

\begin{prop}\label{prop:cond2}
If $A(\lambda)$ satisfies the condition {\rm(\ref{cond2})},
and if $b_1,\dots,b_{n-1}$ and $a_0,\dots,a_n$ are all
distinct, then the following inequalities hold:
\begin{enumerate}
\item \begin{equation*}
\sum_{i=1}^n\int_{a_i^+}^{a_{i-1}^-}
\frac{(-1)^{n-i+\# I}A(\lambda)\prod_{j\in I}(\lambda-b_j)
}
{\sqrt{-\prod_{k=1}^{n-1}(\lambda-b_k)
\cdot\prod_{k=0}^n(\lambda-a_k)}}\ 
d\lambda<0,
\end{equation*}
where $I$ is any (possibly empty)
subset of
$\{1,\dots, n-1\}$ such that $\# I\le n-2$;
\item \begin{equation*}
\frac{\partial }{\partial b_l}
\sum_{i=1}^n\int_{a_i^+}^{a_{i-1}^-}
\frac{(-1)^{i}G_l(\lambda)A(\lambda)\ d\lambda}
{\sqrt{-\prod_{k=1}^{n-1}(\lambda-b_k)
\cdot\prod_{k=0}^n(\lambda-a_k)}}
>0\ ,
\end{equation*}
where $1\le l\le n-1$.
\end{enumerate}
The inequality (1) is still valid if
$b_j$'s $(j\not\in I)$ are mutually distinct.
Precisely speaking, when a sequence of $b_j$'s with 
$b_j$'s and $a_k$'s being all distinct converges to
some $b_j$'s which satisfy $b_k\ne b_l$ for any $k,l
\in J$, $k\ne l$, then the formula in (1) has a limit
and the limit is still negative.
\end{prop}

In the following two lemmas, we shall assume that
$b_1,\dots,b_{n-1}$ and $a_0,\dots,a_n$ are all distinct.
\begin{lemma}\label{lemma:flat}
\begin{equation*}
\sum_{i=1}^n\int_{a_i^+}^{a_{i-1}^-}
\frac{(-1)^{i}G(\lambda)\ d\lambda}
{\sqrt{-\prod_{k=1}^{n-1}(\lambda-b_k)
\cdot\prod_{k=0}^n(\lambda-a_k)}}
=0
\end{equation*}
for any polynomial $G(\lambda)$ of degree $\le n-2$.
\end{lemma}
\begin{proof}
Let $W=\{\lambda\}$ be the region $\C\cup\{\infty\}-\bigcup_{
i=1}^n [a_i^+,a_{i-1}^-]$. Then there are a 
meromorphic function $\mu$ on $W$ such that
\begin{equation*}
 \mu^2=-\prod_{k=1}^{n-1}(\lambda-b_k)
\cdot\prod_{k=0}^n(\lambda-a_k),
\end{equation*}
and the holomorphic
$1$-form $(G(\lambda)/\mu)d\lambda$ on $W$.
Taking the sum of contour integrals around the
intervals $[a_i^+,a_{i-1}^-]$, one obtains the
desired formula.
\end{proof}
\begin{lemma}\label{lemma:cond}
Let $J$ be any nonempty subset of $\{1,\dots,n-1\}$, and
let $B(\lambda)$ be the function defined by
\begin{equation}\label{eq:B}
\frac{A(\lambda)}{\prod_{k\in J}(\lambda-b_k)}=
\sum_{k\in J} \frac{e_k}{\lambda-b_k}+B(\lambda),
\quad e_k=\frac{A(b_k)}{\prod_{\substack{l\in J\\
l\ne k}}(b_k-b_l)}.
\end{equation}
Suppose $A(\lambda)$ satisfies the condition (\ref{cond2}).
Then $B(\lambda)$ satisfies
\begin{equation*}
(-1)^{\# J +m}B^{(m)}(\lambda)<0\quad \text{for}
\quad a_n\le \lambda\le a_0 \quad\text{and}
\quad 0\le m\le n-1-\# J.
\end{equation*}
\end{lemma}
\begin{proof}
We shall prove this by an induction on $\# J$.
When $J=\{k\}$, then
\begin{equation}\label{eq:prime}
B(\lambda)=\frac{A(\lambda)-A(b_k)}{\lambda-b_k}
=\int_0^1A'(t(\lambda-b_k)+b_k)dt,
\end{equation}
and we have $(-1)^{1+m}B^{(m)}(\lambda)<0$ by the 
assumption on $A(\lambda)$.

Now suppose $\# J\ge 1$, $l\not\in J$ and let $J_1=J\cup\{l\}$.
Then
\begin{gather*}
\frac{A(\lambda)}{\prod_{k\in J_1}(\lambda-b_k)}=
\sum_{k\in J} \frac{e_k}{(\lambda-b_k)
(\lambda-b_l)}+\frac{B(\lambda)}{\lambda-b_l}\\
=\sum_{k\in J}\frac1{b_k-b_l}\left(\frac{e_k}
{\lambda-b_k}-\frac{e_k}{\lambda-b_l}\right)+
\frac{B(b_l)}{\lambda-b_l}+
\frac{B(\lambda)-B(b_l)}{\lambda-b_l}.
\end{gather*}
Let us denote the last term in the right-hand side 
by $B_1(\lambda)$. Since it is written as
\begin{equation*}
\int_0^1B'(t(\lambda-b_l)+b_l)dt,
\end{equation*}
we have $(-1)^{\# J +1+m}B_1^{(m)}(\lambda)<0$
by the induction assumption.
\end{proof}

{\it Proof of Proposition \ref{prop:cond2}.}
First, suppose that $b_1,\dots,b_{n-1}$ and $a_0,\dots$, $a_n$ are all distinct.
Let $A(\lambda)$ be a positive function on 
$[a_n,a_0]$ satisfying the condition (\ref{cond2}).
Let $I$ be as in Proposition \ref{prop:cond2} (1) and let $J$ be
its complement in $\{1,\dots,n-1\}$.
Define the function $B(\lambda)$ by the formula
(\ref{eq:B}). Then, by Lemmas \ref{lemma:cond} and
\ref{lemma:flat}
we have
\begin{equation}\label{eq:cond1}
\begin{gathered}
\sum_{i=1}^n\int_{a_i^+}^{a_{i-1}^-}
\frac{(-1)^{n-i+\# I}A(\lambda)\prod_{l\in I}(\lambda-b_l)
}
{\sqrt{-\prod_{k=1}^{n-1}(\lambda-b_k)
\cdot\prod_{k=0}^n(\lambda-a_k)}}\ 
d\lambda\\
=\sum_{i=1}^n\int_{a_i^+}^{a_{i-1}^-}
\frac{(-1)^{n-i+\# I}B(\lambda)\prod_{l=1}^{n-1}
(\lambda-b_l)}
{\sqrt{-\prod_{k=1}^{n-1}(\lambda-b_k)
\cdot\prod_{k=0}^n(\lambda-a_k)}}\ 
d\lambda\ .
\end{gathered}
\end{equation}
Since $(-1)^{i-1}\prod_{j=1}^{n-1}(\lambda-b_j)>0$
on $(a_i^+,a_{i-1}^-)$, and since
\begin{equation*}
(-1)^{n-1-\# I}B(\lambda)<0
\end{equation*}
by Lemma \ref{lemma:cond}, we have the inequality 
(1) in this case.

Next, let us consider the limit case. The limit $b_j$'s are assumed that $b_k\ne b_l$ for any $k,l
\in J$, $k\ne l$. Note that the function $B(\lambda)$
is defined by the formula (\ref{eq:B}) and it only
depends on $A(\lambda)$ and $b_j$'s $(j\in J)$.
Since the limit $b_j$'s $(j\in J)$ are mutually 
distinct, it follows that the function $B(\lambda)$
has a limit. Therefore the right-hand side of the
formula (\ref{eq:cond1}) has a finite limit and it is
still negative by the same reason as above.

To prove (2), we put
\begin{equation*}
\frac{A(\lambda)}{\lambda-b_l}=
\frac{A(b_l)}{\lambda-b_l}+B(\lambda, b_l).
\end{equation*}
Then the left-hand side of (2) is equal to
\begin{equation}\label{eq:cond}
\begin{gathered}
\frac{\partial }{\partial b_l}
\sum_{i=1}^n\int_{a_i^+}^{a_{i-1}^-}
\frac{(-1)^{i}B(\lambda, b_l)\prod_{j=1}^{n-1}(\lambda-b_j)}
{\sqrt{-\prod_{k=1}^{n-1}(\lambda-b_k)
\cdot\prod_{k=0}^n(\lambda-a_k)}}\ 
d\lambda\\
=\sum_{i=1}^n\int_{a_i^+}^{a_{i-1}^-}
\frac{(-1)^{i}\left(\frac{\partial }
{\partial b_l}B(\lambda, b_l)\right)
\prod_{j=1}^{n-1}(\lambda-b_j)}
{\sqrt{-\prod_{k=1}^{n-1}(\lambda-b_k)
\cdot\prod_{k=0}^n(\lambda-a_k)}}\ 
d\lambda\\
-\frac12\sum_{i=1}^n\int_{a_i^+}^{a_{i-1}^-}
\frac{(-1)^{i}B(\lambda, b_l)\prod_{\substack{
1\le j\le n-1\\j\ne l}}(\lambda-b_j)}
{\sqrt{-\prod_{k=1}^{n-1}(\lambda-b_k)
\cdot\prod_{k=0}^n(\lambda-a_k)}}\ 
d\lambda.
\end{gathered}
\end{equation}
The second line of the right-hand side is equal to
\begin{equation*}
-\frac12\sum_{i=1}^n\int_{a_i^+}^{a_{i-1}^-}
\frac{(-1)^{i}B_1(\lambda, b_l)\prod_{
1\le j\le n-1}(\lambda-b_j)}
{\sqrt{-\prod_{k=1}^{n-1}(\lambda-b_k)
\cdot\prod_{k=0}^n(\lambda-a_k)}},
\end{equation*}
where 
\begin{equation*}
B_1(\lambda,b_l)=\frac{B(\lambda,b_l)-A'(b_l)}
{\lambda-b_l}=\frac{\partial}{\partial b_l}B(\lambda,b_l).
\end{equation*}
Since $B_1(\lambda,b_l)<0$, it follows that the right-hand
side of the formula (\ref{eq:cond}) is positive.
\qed

%

\section{Jacobi fields}\label{sec:jac}
In this section we shall consider Jacobi fields 
along a geodesic which is not totally contained
in the submanifold $N_i$ for any $i$.
Let $\gamma(t)=(x_1(t),\dots,x_n(t))$ be such a 
geodesic. In this case, the corresponding 
values $b_i$ of the first integrals
$H_i$ satisfy
$b_i\ne a_{i+1}$ and $b_i\ne a_{i-1}$ for any $i$.
We shall consider the following
three cases separately: (i) $b_1,\dots,b_{n-1}$ and
$a_0,\dots,a_n$ are all distinct;
(ii) there are some $i$ such that $b_i=a_i$, but
other $b_j$'s are not equal to any $a_k$ nor 
$b_k$; (iii) there are some $j$ such that
$b_j=b_{j-1}$, and there may be some $i$ such that 
$b_i=a_i$, but there is no $l$ such that
$b_l=a_{l+1}$ or $b_l=a_{l-1}$.

First, let us consider the case where $b_1,
\dots,b_{n-1}$ and $a_0,\dots,a_n$ are all
distinct.
For each $i$, let $S_i\subset \R$ be the set of the time $s$
such that $f_i(x_i(s))=b_i$ $(b_i=a_i^+)$ or 
$f_{i+1}(x_{i+1}(s))=b_i$ $(b_i=a_i^-)$.
Then $S_i$ are discrete subsets of $\R$. At each point $\gamma(s)$ where $s\not\in S_i$ 
for any $i$, the system of functions $(H_1,\dots,H_{n-1})$ 
can be used
as a coordinate system on the unit cotangent space 
$U^*_{\gamma(s)}M$ around the covector $(x(s),\xi(s))=
\flat(\dot\gamma(s))$. Then, identifying $\partial/\partial
H_i\in T_{\flat(\dot\gamma(s))}(U^*_{\gamma(s)}M)$ with
a covector in $T^*_{\gamma(s)}M$ in a natural manner,
we put $\tilde V_i(s)=\sharp(\frac{\partial}{\partial H_i}
/|\frac{\partial}{\partial H_i}|)\in T_{\gamma(s)}M$ at 
$\gamma(s)$. As is easily seen, the norm $|\partial/\partial 
H_i|$ is equal to
\begin{equation*}
\frac12\sqrt{\frac{(-1)^{n-1}G_i(b_i)}{\prod_{m=1}^n(f_m(x_m)-b_i)}}\ .
\end{equation*}

At the point $\gamma(s)$ where $s\in S_i$, we put $\nu_i^2=f_i(x_i(s))-H_i$ if $b_i=a_i^+$ (resp. 
$\nu_i^2=H_i-f_{i+1}(x_{i+1}(s))$ if $b_i=a_i^-$),
and use $\nu_i$ as a coordinate function on 
$U^*_{\gamma(s)}M$ 
instead of $H_i$. 
We choose the sign of $\nu_i$ so that it is equal 
with the sign of $\xi_i$ (resp. $\xi_{i+1}$). Then we put $\tilde V_i(s)=
\sharp(\frac{\partial}
{\partial \nu_i}/|\frac{\partial}{\partial \nu_i}|)$
in this case. 
It is easy to see that $\R\ni s\mapsto \tilde V_i(s)$ is
smooth up to the sign. Therefore we can take a smooth
vector field
$V_i(t)$ along the geodesic $\gamma(t)$ such that
$V_i(t)= \pm\tilde V_i(t)$ for any $t\in \R$.
We now define the Jacobi field $Y_{i, s}(t)$ along the geodesic
$\gamma(t)$ by the initial conditions $Y_{i,s}(s)=0$ and 
$Y'_{i,s}(s)=V_i(s)$ for any $s\in \R$, where 
$Y'_{i,s}(t)$ denotes the covariant derivative of
$Y_{i,s}(t)$ with respect to $\partial/\partial t$.

Let us denote by $\Omega(Y,Z)$ the symplectic inner product 
of two Jacobi fields along $\gamma(t)$ which are
orthogonal to $\dot\gamma(t)$ for any $t$:
\begin{equation*}
\Omega(Y,Z)=g(Y(t), Z'(t))-g(Y'(t),Z(t))\ ,
\end{equation*}
which is constant in $t$.
Let $\mathcal Y_i$ be the vector space of Jacobi fields along
$\gamma(t)$ spanned by $\{Y_{i,s}(t)\ |\ s\in \R\}$.

\begin{prop}\label{prop:jf1}
Along the geodesic $\gamma(t)$ such that $b_1,\dots,b_{n-1}$ 
and $a_0, \dots, a_n$ are all distinct, the Jacobi fields 
defined above have the following properties.
\begin{enumerate}
\item $Y_{i,s}(t)\in\R V_i(t)$ for any $i$ and $s,t\in\R$.
Also, $V_1(t),\dots,V_{n-1}(t)$, $\dot\gamma(t)$ are
mutually orthogonal for any $t\in\R$.
\item $\mathcal Y_i$ and $\mathcal Y_j$ $(i\ne j)$ are mutually orthogonal 
with respect to the symplectic inner product $\Omega$, 
i.e., $\Omega(Y_i,Y_j)=0$
for any $Y_i\in \mathcal Y_i$ and $Y_j\in\mathcal Y_j$.
\item Each $V_i(t)$ is parallel along the geodesic 
$\gamma(t)$.
\item Each $\mathcal Y_i$ is two-dimanesional.
\item If $\gamma(s_1)$ and $\gamma(s_2)$ $(s_1<s_2)$ 
are mutually
conjugate along the geodesic $\gamma(t)$, then there is
$i$ and a nonzero Jacobi field $Y\in\mathcal Y_i$ such that
$Y(s_1)=Y(s_2)=0$.
\item $Y_{i,s_1}(s_2)\ne 0$ if $s_1\not\in S_i$,
$s_2\ne s_1$, and either $[s_1,s_2)\cap S_i=\emptyset$, $s_1<s_2$ or $(s_2,s_1]\cap S_i=\emptyset$, $s_2<s_1$.
\item The Jacobi field $Y_{i,s_1}(t)$ $(s_1\in S_i)$
vanishes at $t=s_2$ if and only if $s_2\in S_i$.
\end{enumerate}
\end{prop}
\begin{proof}
Let $\gamma(u,t)=(\dots,x_k(u,t),\dots)$ be a one-parameter 
family of geodesics such that $x_k(0,t)=x_k(t)$ and 
$(\partial/\partial u)|_{u=0}$ represents
the Jacobi field $Y_{i,s_1}(t)$. Suppose that $G=G_j$, 
$i\ne j$, and $s=s_1$ and $t=s_2$ do not 
belong to $S_i\cup S_j$ in the formula
(\ref{eq:geodsigma1}).
 We then differentiate the 
formula by $u$. Since
\begin{equation*}
\frac{\partial H_k}{\partial u}\big|_{u=0}\ne 0\quad (k=i)\ ;\quad
=0 \quad (k\ne i)\ ,
\end{equation*}
we have
\begin{equation}
\begin{gathered}
\sum_{l=1}^n\frac{\epsilon'_l(-1)^lG_j(f_l)A(f_l)}
{\sqrt{-\prod_{k=1}^{n-1}(f_l-b_k)
\cdot\prod_{k=0}^n(f_l-a_k)}}\,d(f_l(x_l))(Y_{i,s_1}(s_2))\\
-\frac1{2c}\sum_{l=1}^n\int_{\sigma_l(s_1)}^{\sigma_l(s_2)}\frac{(-1)^iG_{i,j}(f_l)A(f_l)}
{\sqrt{-\prod_{k=1}^{n-1}(f_l-b_k)
\cdot\prod_{k=0}^n(f_l-a_k)}}\ 
\ d\sigma_l=0\ ,
\end{gathered}
\end{equation}
where $c=\pm$ (the norm of $\partial/\partial H_i$ at 
$\gamma(s_1)$) and $f_l=f_l(x_l(s_2))$ in the first line, and 
$G_{i,j}(\lambda)=\prod_{k\ne i,j}(\lambda-b_k)$.
Observe that the second line in the above formula vanishes
by the formula (\ref{eq:geodsigma1}). Moreover, the 
covector
\begin{equation*}
\frac14\sum_{l=1}^n\frac{\epsilon'_l(-1)^lG_j(f_l)A(f_l)}
{\sqrt{-\prod_{k=1}^{n-1}(f_l-b_k)
\cdot\prod_{k=0}^n(f_l-a_k)}}\ 
d(f_l(x_l))\bigg|_{f_l=f_l(x_l(s_2))}
\end{equation*}
is equal to the one which is represented by $\partial
/\partial H_j$ at $\gamma(s_2)$, which is a nonzero scalar
multiple of
$\flat(Y'_{j,s_2}(s_2))$. Thus we have
\begin{equation*}
\Omega(Y_{i,s_1},Y_{j,s_2})=
g(Y_{i,s_1}(s_2),Y'_{j,s_2}(s_2))=0\ ,
\end{equation*}
which is valid for any $s_1,s_2\in\R$ by continuity. In 
particular, we have $g(Y_{i,s_1}(s_2),V_j(s_2))=0$ for any 
$j\ne i$, and also $g(V_i(s_1),V_j(s_1))=0$ by differentiating it
at $s_2=s_1$. Thus we have (1) and (2).

(3) and (4) follow immediately from (1) and (2).
The assertion (5) is also obvious.
Next, we shall prove (6). First, we assume $s_1<s_2$ and
$s_2\not\in S_i$. 
In the same way as above, we have
\begin{equation}\label{eq:jac4}
\begin{gathered}
\sum_{l=1}^n\frac{\epsilon'_l(-1)^lG_i(f_l)A(f_l)}
{\sqrt{-\prod_{k=1}^{n-1}(f_l-b_k)
\cdot\prod_{k=0}^n(f_l-a_k)}}\,
d(f_l(x_l))(Y_{i,s_1}(s_2))\\
+\frac1{2c}\sum_{l=1}^n\int_{\sigma_l(s_1)}^{\sigma_l(s_2)}\frac{(-1)^iG_{i}(f_l)A(f_l)}
{(f_l-b_i)\sqrt{-\prod_{k=1}^{n-1}(f_l-b_k)
\cdot\prod_{k=0}^n(f_l-a_k)}}\,
d\sigma_l=0\ .
\end{gathered}
\end{equation}
Note that, since $[s_1,s_2]\cap S_i=\emptyset$, $f_l-b_i$
never vanish on the interval $[\sigma_l(s_1),\sigma_l(s_2)]$.
The second line in the above formula being negative, we have
$g(Y_{i,s_1}(s_2),Y'_{i,s_2}(s_2))\ne 0$. Thus
$Y_{i,s_1}(s_2)\ne 0$.

Next, let us take $s_3\in S_i$ such that $s_1<s_3$ and $[s_1,s_3)\cap S_i=\emptyset$. As proved above, 
\begin{gather*}
\left|\frac{\partial }{\partial H_i}\right|_{\gamma(s_1)}
\left|\frac{\partial }{\partial H_i}\right|_{\gamma(s_2)}
g(Y_{i,s_1}(s_2),Y'_{i,s_2}(s_2))=\\
-\frac18\sum_{l=1}^n\int_{\sigma_l(s_1)}^{\sigma_l(s_2)}\frac{(-1)^iG_{i}(f_l)A(f_l)}
{(f_l-b_i)\sqrt{-\prod_{k=1}^{n-1}(f_l-b_k)
\cdot\prod_{k=0}^n(f_l-a_k)}}\ 
\ d\sigma_l
\end{gather*}
for any $s_2$ such that $s_1<s_2<s_3$. Suppose $b_i=a_i^+$.
Since
\begin{equation*}
g(Y_{i,s_1}(s_2),Y'_{i,s_2}(s_2))=\Omega(Y_{s_1},Y_{s_2})
=-g(Y'_{i,s_1}(s_1),Y_{i,s_2}(s_1))\,,
\end{equation*}
multiplying both sides by $2|\nu_i|=2\sqrt{f_i(x_i(s_2))-b_i}$, 
and taking a limit $s_2\to s_3$, we have
\begin{equation}\label{eq:singjac2}
-c'g(Y'_{i,s_1}(s_1), Y_{i,s_3}(s_1))=
\frac12\frac{(-1)^{i+1}G_{i}(b_i)A(b_i)}
{\sqrt{-\prod_{k\ne i}(b_i-b_k)
\cdot\prod_{k=0}^n(b_i-a_k)}}\ ,
\end{equation}
where $c'=|\partial/\partial H_i|_{\gamma(s_1)}
|\partial/\partial \nu_i|_{\gamma(s_3)}$.
Since the left-hand side of the above formula is equal to
\begin{equation*}
c'g(Y_{i,s_1}(s_3), Y'_{i,s_3}(s_3))\ , 
\end{equation*}
and since
the right-hand side does not vanish, we have
\begin{equation}\label{eq:singjac3}
Y_{i,s_1}(s_3)\ne 0\ ,\qquad
Y_{i,s_3}(s_1)\ne 0\ .
\end{equation}
The case where $s_2<s_1$ 
is similar. Therefore the assertion (6) follows.

Now, in the situation of (6), take $s_0\in S_i$ such that
$s_0<s_1$ and $(s_0,s_1]\cap S_i=\emptyset$. Then, again 
multiplying both sides of the formula (\ref{eq:singjac2})
by $|\nu_i|=\sqrt{f_i(x_i(s_1))-b_i}$ and taking a limit
$s_1\to s_0$, we have
\begin{equation*}
g(Y_{i,s_0}(s_3),Y'_{i,s_3}(s_3))=0\ .
\end{equation*}
Thus it follows that $Y_{i,s_0}(s_3)=0$, and 
combined with (\ref{eq:singjac3}) we have (7).
\end{proof}
The following corollary is immediate.
\begin{cor}\label{cor:conj}
Fix $t_0$ and let $t_0<t_1^i<t_2^i<\dots$ be the zeros of
the Jacobi field $Y_{i,t_0}(t)$ for $t\ge t_0$. Then:
\begin{enumerate}
\item If $t_0\in S_i$, then the set $\{t_k^i\}$
coincides with $\{t\in S_i\ |\ t>t_0\}$
\item If $t_0\not\in S_i$, then every $t_k^i\not\in S_i$,
and there is just one element of $S_i$ in the interval $(t_k^i,
t_{k+1}^i)$ for each $k$.
\item The set  of conjugate points of $\gamma(t_0)$ along
$\gamma(t)$ $(t>t_0)$ is equal to $\{\gamma(t_k^i) \ |
\ k\ge1, 1\le i\le n-1\}$.
\end{enumerate}
\end{cor}

We shall prove one more result on the zeros of Jacobi fields
in this case, which needs
the assumption \eqref{cond2}.
\begin{prop}\label{prop:jacreg}
Fix $i$ and take $s_1$ and $s_2$ such that $s_1\not\in S_i$, $s_1<s_2$, and $\sigma_l(s_2)
-\sigma_l(s_1)\le 2(a_{l-1}^--a_l^+)$ for any $l$. Then $Y_{i,s_1}(s_2)
\ne 0$.
\end{prop}
\begin{proof}
Let $s_3\in S_i$ such that $s_1<s_3$ and $[s_1,s_3)\cap S_i=\emptyset$.
If $s_2\le s_3$, then the assertion follows from (5) of
the previous proposition. Now suppose $s_3<s_2$.
As above, we shall compute $g(Y_{i,s_1}(s_2),Y'_{i,s_2}(s_2))$. In this case, however,
the formula \eqref{eq:jac4} is invalid, because the integral
diverge at $t=s_3$. So, instead, we differentiate the 
formula
\begin{equation}\label{eq:geod3}
\begin{aligned}
&-\sum_{l=1}^{n}\int_{\sigma_l(s_2)}^{
2(a_{l-1}^--a_l^+)+\sigma_l(s_1)}\frac{(-1)^lG_i(f_l)
A(f_l)\ d\sigma_l}
{\sqrt{-\prod_{k=1}^{n-1}(f_l-b_k)
\cdot\prod_{k=0}^n(f_l-a_k)}}\\
&+2\sum_{l=1}^n\int_{a_l^+}^{a_{l-1}^-}
\frac{(-1)^lG_i(\lambda)
A(\lambda)\ d\lambda}
{\sqrt{-\prod_{k=1}^{n-1}(\lambda-b_k)
\cdot\prod_{k=0}^n(\lambda-a_k)}}=0
\end{aligned}
\end{equation}
in terms of the deformation parameter defining $cY_{i,s_1}$,
$c$ being $\pm$ (the norm of $\partial/\partial H_i$ at 
$\gamma(s_1)$):
\begin{equation}\label{eq:diff3}
\begin{aligned}
&\sum_{l=1}^n\frac{\epsilon'_l(-1)^lG_i(f_l)A(f_l)}
{\sqrt{-\prod_{k=1}^{n-1}(f_l-b_k)
\cdot\prod_{k=0}^n(f_l-a_k)}}\ 
d(f_l(x_l))(cY_{i,s_1}(s_2))\\
-&\frac1{2}\sum_{l=1}^{n}\int_{\sigma_l(s_2)}^{
2(a_{l-1}^--a_l^+)+\sigma_l(s_1)}\frac{(-1)^lG_i(f_l)
A(f_l)\ d\sigma_l}
{(f_l-b_i)\sqrt{-\prod_{k=1}^{n-1}(f_l-b_k)
\cdot\prod_{k=0}^n(f_l-a_k)}}\\
+&2\frac{\partial}{\partial b_i}
\sum_{l=1}^n\int_{a_l^+}^{a_{l-1}^-}
\frac{(-1)^lG_i(\lambda)
A(\lambda)\ d\lambda}
{\sqrt{-\prod_{k=1}^{n-1}(\lambda-b_k)
\cdot\prod_{k=0}^n(\lambda-a_k)}}=0\ ,
\end{aligned}
\end{equation}

Note that
$b_i$ is not contained in the range of $f_l$
while $\sigma_l$ moves in the interval $[\sigma_l(s_2),
2(a_{l-1}^--a_l^+)+\sigma_l(s_1)]$ $(l=i,i+1)$.
Since the second line of the formula \eqref{eq:diff3}
is positive or zero, and since the third line is positive
by Proposition \ref{prop:cond2} (2), it therefore follows
that $g(Y_{i,s_1}(s_2),Y'_{i,s_2}(s_2))\ne 0$.
\end{proof}
Next, we shall consider Jacobi fields along the geodesic
$\gamma(t)$ for which some $b_i$ is equal to $a_i$, but 
other $b_j$'s are not equal to any $a_k$ nor 
$b_k$.
For $i$ with 
$b_i=a_i$, let $S_i$ be the set of $s\in\R$ where
$f_i(x_i(s))=b_i$. One can see from the formula 
(\ref{eq:geodsigma1}) that $S_i$ is also the set of $s\in\R$
where $f_{i+1}(x_{i+1}(s))=b_i$, i.e., $s\in S_i$ if
and only if $\gamma(s)\in J_i$. For such $i$ and $s\in S_i$,
we define $\tilde Y_{i,s}(t)$ as the Jacobi field
$\pi_*(X_{F_i})$ along the geodesic $\gamma(t)$.
For $s\not\in S_i$, $Y_{i,s}(t)$ is defined as before.
Also, for $j$ with $b_j\ne a_j$, the set $S_j$
and the Jacobi fields $Y_{j,s}(t)$
are defined as before.
\begin{prop}
For a geodesic $\gamma(t)$ stated above, the statements in
Propositions \ref{prop:jf1}, \ref{prop:jacreg} 
and Corollary \ref{cor:conj} equally hold.
\end{prop}
\begin{proof}
Only the parts related to the Jacobi field 
$\tilde Y_{i,s}(t)=\pi_*(X_{F_i})$ would be nontrivial.
Suppose $b_i=a_i$ and $s_1\not\in S_j$, $s_2\in S_i$.
Considering the symplectic inner product of two Jacobi fields
$Y_{j,s_1}(t)$ and $\tilde Y_{i,s_2}(t)$, we have
\begin{gather*}
\Omega(Y_{j,s_1},\tilde Y_{i,s_2})=
c\,\omega\left(\frac{\partial}{\partial H_j},X_{F_i}
\right)_{\flat(\dot\gamma(s_1))}\\
= c\,\frac{\partial c_i}{\partial b_j}=
\frac{c\,\prod_{m\ne j}(a_i-b_m)}
{\prod_{\substack{1\le k\le n-1\\k\ne i}}(a_i-a_k)}\quad
\begin{cases}
=0\quad (j\ne i)\\
\ne 0\quad (j=i)
\end{cases}\ ,
\end{gather*}
where $\omega$ is the symplectic 2-form $\sum_kd\xi_k\wedge
dx_k$, $\partial/\partial H_j$ is the tangent vector
to $U^*_{\gamma(s_1)}M$ at $\flat(\dot\gamma(s_1))$ defined
as before, and $c=1/|\partial/\partial H_j|$.
The proposition follows from this formula.
\end{proof}

Next, we shall consider Jacobi fields along a 
geodesic for which there are some $j$ such that
$b_j=b_{j-1}$ and there may be some $i$ such that 
$b_i=a_i$, but there is no $l$ such that
$b_l=a_{l+1}$ or $b_l=a_{l-1}$. 
In this case, 
$f_j(x_j(t))(=b_j=b_{j-1})$
remains constant along the geodesic $\gamma(t)$. We put
this value $\lambda_j^0$ for convenience.
For each point $\gamma(s)$ on the geodesic,
we adopt $\mu_j,
\mu_{j-1}$ as the coordinate functions on the
unit cotangent space $U^*_{\gamma(s)}M$,
around the covector $\flat(\dot\gamma(s))$,
instead of $H_j, H_{j-1}$, defined by the 
formula:
\begin{equation*}
\mu_{j-1}=H_{j-1}+H_j-2\lambda_j^0,\quad
\mu_j^2=4(H_{j-1}-\lambda_j^0)(\lambda_j^0-H_j)
\ .
\end{equation*}
We choose the sign of $\mu_j$ so that it is equal to that of
$\xi_j$. Let us denote by $Z_{j,s}(t)$, 
$Z_{j-1,s}(t)$ the Jacobi fields along the 
geodesic $\gamma(t)$ with the initial conditions
\begin{equation*}
Z_{k,s}(s)=0,\ Z_{k,s}'(s)=\sharp(\partial/
\partial \mu_k)/|\partial/\partial \mu_k|
\quad (k=j,j-1)\ .
\end{equation*}
Note that
\begin{equation*}
\left|\frac{\partial}{\partial \mu_{j-1}}\right|=
\left|\frac{\partial}{\partial \mu_{j}}\right|=
\frac12\sqrt{\frac{(-1)^nG_{j,j-1}(\lambda_j^0)}
{\prod_{m\ne j}(f_m-\lambda_j^0)}},\quad
\left\langle\frac{\partial}{\partial \mu_{j-1}},
\frac{\partial}{\partial \mu_{j}}\right\rangle=0
\end{equation*}
at each covector $\flat(\dot\gamma(s))$.

Define the real number $\theta_{s_1}(s_2)$ by 
the formula
\begin{equation}\label{eq:theta}
\begin{gathered}
\sum_{\substack{1\le l\le n\\l\ne j}}\int_{\sigma_l(s_1)}^{\sigma_l(s_2)}\frac{(-1)^l
G_{j,j-1}(f_l)A(f_l)\ d\sigma_l}
{|f_l-\lambda_j^0|\sqrt{-\prod_{k\ne j,j-1}(f_l-b_k)
\cdot\prod_{k=0}^n(f_l-a_k)}}\\
+2\theta_{s_1}(s_2)\ \frac{(-1)^jG_{j,j-1}(\lambda_j^0)A(\lambda_j^0)}
{\sqrt{\prod_{k\ne j,j-1}(\lambda_j^0-b_k)
\prod_k(\lambda_j^0-a_k)}}\ =0\ .
\end{gathered}
\end{equation}
We then have the following proposition.
\begin{prop}
\begin{enumerate}
\item
$Z_{k,s_1}(s_2)=0$ for $k=j,j-1$ and any 
$s_1,s_2$ such that $\theta_{s_1}(s_2)=\pi$.
\item $Z_{j,s_1}(s_2)$ and $Z_{j-1,s_1}(s_2)$
are linearly independent for any $s_1$ and $s_2$
such that $0<\theta_{s_1}(s_2)<\pi$.
\end{enumerate}
\end{prop}
\begin{proof}
We consider a one-parameter family of geodesics $t\to\gamma(u,t)$ such that $\gamma(0,t)=\gamma(t)$, 
$\gamma(u,s_1)=\gamma(s_1)$, and the values $b_i$
of the first integrals $H_i$ for $\gamma(u,t)$ are the same as those 
for $\gamma(t)$ except that $b_{j-1}(u)=
H_{j-1}(\flat(\dot\gamma(u,t)))=\lambda_j^0+u^2$.
Since $b_j=\lambda_j^0=f_j(x_j(u,s_1))$ for any $u$, it 
follows that the Jacobi fields $Y_{j,s_1}(t)$ and
$Y_{j-1,s_1}(t)$ are defined along the geodesic 
$\gamma(u,t)$ for $u\ne 0$. Observe that on the
unit cotangent space $U^*_{\gamma(s_1)}M$, 
$(\partial/\partial \nu_j)/|\partial/\partial \nu_j|$ tends to 
$\pm(\partial/\partial \mu_j)/|\partial/\partial \mu_j|$
and $(\partial/\partial H_{j-1})/|\partial/\partial H_{j-1}|$ tends to 
$(\partial/\partial\mu_{j-1})/|\partial/
\partial\mu_{j-1}|$ as $u\to 0$. Thus the Jacobi 
fields 
$Y_{j,s_1}(t)$ and $Y_{j-1,s_1}(t)$ along
the geodesic $\gamma(u,t)$ converge to Jacobi fields 
$Z_{j,s_1}(t)$ and $Z_{j-1,s_1}(t)$ up to the sign along the 
geodesic $\gamma(t)$ as $u\to 0$.

Moreover, with this procedure of taking the limit, we claim 
that the Jacobi fields $Y_{j,s_2}(t)$ and $Y_{j-1,s_2}(t)$
along the geodesic $\gamma(u,t)$ tend to 
\begin{equation*}
\epsilon\left(\cos\theta Z_{j,s_2}(t)+
\sin\theta Z_{j-1,s_2}(t)\right)\ 
\text{and}\ \epsilon
\left(-\sin\theta Z_{j,s_2}(t)+
\cos\theta Z_{j-1,s_2}(t)\right)
\end{equation*}
respectively, where $\epsilon=\pm 1$ and 
$\theta=\theta_{s_1}(s_2)$.
To see this, we begin with the formula before taking the 
limit:
\begin{equation}
\sum_{i=1}^n\int_{\sigma_i(s_1)}^{\sigma_i(s_2)}\frac{(-1)^iG_{j,j-1}(f_i)A(f_i)}
{\sqrt{-\prod_{k=1}^{n-1}(f_i-b_k)
\cdot\prod_{k=0}^n(f_i-a_k)}}\ 
\ d\sigma_i=0\ .
\end{equation}
Define the function $\theta(u,t)$ by
\begin{gather*}
f_j(x_j(u,t))=b_j(\cos\theta(u,t))^2+b_{j-1}(u)
(\sin\theta(u,t))^2\ ,\\
\theta(u,s_1)=0,\qquad (\partial/\partial t)\theta\ge 0\ .
\end{gather*}
Then, taking the limit $u\to 0$, 
we see that
\begin{equation*}
\int_{\sigma_j(s_1)}^{\sigma_j(s_2)}\frac{(-1)^jG_{j,j-1}(f_j)A(f_j)}
{\sqrt{-\prod_{k=1}^{n-1}(f_j-b_k)
\cdot\prod_{k=0}^n(f_j-a_k)}}\ 
\ d\sigma_j
\end{equation*}
tends to 
\begin{equation*}
2\theta(0,s_2)\ \frac{(-1)^jG_{j,j-1}(\lambda_j^0)A(\lambda_j^0)}
{\sqrt{\prod_{k\ne j,j-1}(\lambda_j^0-b_k)
\prod_k(\lambda_j^0-a_k)}}\ .
\end{equation*}
Thus we have $\theta(0,t)=\theta_{s_1}(t)$ by (\ref{eq:theta}). The covector $\partial/\partial H_j$
at the point $\gamma(u,s_2)$ is equal to
\begin{equation*}
\frac14\sum_{i=1}^n\frac{\epsilon'_i(-1)^iG_j(f_i)A(f_i)
\ df_i}
{\sqrt{-\prod_{k=1}^{n-1}(f_i-b_k)
\cdot\prod_{k=0}^n(f_i-a_k)}}\ ,
\end{equation*}
which tends to, as $u\to 0$,
\begin{gather*}
\frac14\sum_{i\ne j}\frac{f_i-\lambda_j^0}
{|f_i-\lambda_j^0|}\frac{\epsilon'_i(-1)^i
G_{j,j-1}(f_i)A(f_i)\ df_i}
{\sqrt{-\prod_{k\ne j,j-1}(f_i-b_k)
\cdot\prod_{k=0}^n(f_i-a_k)}}\\
+\frac14\frac{(-1)^{j+1}\cot \theta\ G_{j,j-1}
(\lambda_j^0)A(\lambda_j^0)
\ df_j}
{\sqrt{\prod_{k\ne j,j-1}(\lambda_j^0-b_k)
\cdot\prod_{k=0}^n(\lambda_j^0-a_k)}}\ ,
\end{gather*}
where $\theta=\theta_{s_1}(s_2)$. Also, 
$\partial/\partial H_{j-1}$ tends to
\begin{gather*}
\frac14\sum_{i\ne j}\frac{f_i-\lambda_j^0}
{|f_i-\lambda_j^0|}\frac{\epsilon'_i(-1)^i
G_{j,j-1}(f_i)A(f_i)\ df_i}
{\sqrt{-\prod_{k\ne j,j-1}(f_i-b_k)
\cdot\prod_{k=0}^n(f_i-a_k)}}\\
+\frac14\frac{(-1)^{j}\tan \theta\ G_{j,j-1}
(\lambda_j^0)A(\lambda_j^0)
\ df_j}
{\sqrt{\prod_{k\ne j,j-1}(\lambda_j^0-b_k)
\cdot\prod_{k=0}^n(\lambda_j^0-a_k)}}\ ,
\end{gather*}
As is easily seen, we have
\begin{gather*}
\flat(Z'_{j-1,s_2}(s_2))
=\frac{c}4\sum_{i\ne j}\frac{f_i-\lambda_j^0}
{|f_i-\lambda_j^0|}\frac{\epsilon'_i(-1)^i
G_{j,j-1}(f_i)A(f_i)\ df_i}
{\sqrt{-\prod_{k\ne j,j-1}(f_i-b_k)
\cdot\prod_{k=0}^n(f_i-a_k)}}\\
\flat(Z'_{j,s_2}(s_2))=\frac{c}4\frac{(-1)^{j+1}G_{j,j-1}
(\lambda_j^0)A(\lambda_j^0)
\ df_j}
{\sqrt{\prod_{k\ne j,j-1}(\lambda_j^0-b_k)
\cdot\prod_{k=0}^n(\lambda_j^0-a_k)}}\ ,
\end{gather*}
where $c=1/|\partial/\partial \mu_{j-1}|=
1/|\partial/\partial \mu_{j}|$ at $\gamma(s_2)$.
Therefore the claim follows. 

From the formulas obtained above and 
(\ref{eq:singjac2}), we thus have
\begin{equation}\label{eq:z1}
\begin{gathered}
g\left(Z_{j-1,s_1}(s_2),\ \cos\theta\ Z'_{j,s_2}(s_2)+
\sin\theta\ Z'_{j-1,s_2}(s_2)\right)=0\ ,\\
g\left(Z_{j,s_1}(s_2),\ -\sin\theta\ Z'_{j,s_2}(s_2)+
\cos\theta\ Z'_{j-1,s_2}(s_2)\right)=0\ ,\\
g\left(Z_{j,s_1}(s_2),\ \cos\theta\ Z'_{j,s_2}(s_2)+
\sin\theta\ Z'_{j-1,s_2}(s_2)\right)\\
=\frac{\sin\theta}{4cc'}\frac{(-1)^{j}G_{j,j-1}
(\lambda_j^0)A(\lambda_j^0)}
{\sqrt{-\prod_{k\ne j,j-1}(\lambda_j^0-b_k)
\cdot\prod_{k=0}^n(\lambda_j^0-a_k)}}\ ,
\end{gathered}
\end{equation}
where $c$ and $c'$ are the norms of $\partial/\partial \mu_j$ at $\gamma(s_1)$ and $\gamma(s_2)$
respectively.
In particular, we have:
\begin{align*}
&\cos\theta\ \Omega(Z_{j-1,s_1},Z_{j,s_2})
+\sin\theta\ \Omega(Z_{j-1,s_1},Z_{j-1,s_2})=0\\
&-\sin\theta\ \Omega(Z_{j,s_1},Z_{j,s_2})
+\cos\theta\ \Omega(Z_{j,s_1},Z_{j-1,s_2})=0\ ,
\end{align*}
where $\theta=\theta_{s_1}(s_2)$. As is easily seen, the above
formula is also valid when $s_2<s_1$, in which case 
$\theta_{s_1}(s_2)=-\theta_{s_2}(s_1)<0$. Therefore, 
exchanging $s_1$ and $s_2$ in the above formula, we have
\begin{equation}\label{eq:z2}
\begin{aligned}
&\Omega(Z_{j,s_1},Z_{j,s_2})=
\Omega(Z_{j-1,s_1},Z_{j-1,s_2})\\
& \Omega(Z_{j-1,s_1},Z_{j,s_2})
=- \Omega(Z_{j,s_1},Z_{j-1,s_2})\ .
\end{aligned}
\end{equation}
By (\ref{eq:z1}) and (\ref{eq:z2}) we also have
\begin{equation}\label{eq:z3}
\begin{aligned}
g\left(Z_{j-1,s_1}(s_2),\ -\sin\theta\ Z'_{j,s_2}(s_2)+
\cos\theta\ Z'_{j-1,s_2}(s_2)\right)\\
=\frac{\sin\theta}{4cc'}\frac{(-1)^{j}G_{j,j-1}
(\lambda_j^0)A(\lambda_j^0)}
{\sqrt{-\prod_{k\ne j,j-1}(\lambda_j^0-b_k)
\cdot\prod_{k=0}^n(\lambda_j^0-a_k)}}\ .
\end{aligned}
\end{equation}
Now the assertion (2) easily follows from (\ref{eq:z1}) and 
(\ref{eq:z3}). Also, from those formulas we have
\begin{align*}
&g(Z_{j,s_1}(s_2),Z'_{j,s_2}(s_2))=
g(Z_{j,s_1}(s_2),Z'_{j-1,s_2}(s_2))=0\\
& g(Z_{j-1,s_1}(s_2),Z'_{j,s_2}(s_2))=
g(Z_{j-1,s_1}(s_2),Z'_{j-1,s_2}(s_2))=0\ ,
\end{align*}
provided $\theta_{s_1}(s_2)=\pi$. Since the Jacobi fields
$Z_{j,s}$, $Z_{j-1,s}$ belong to the limit of the vector
space $\mathcal Y_j+\mathcal Y_{j-1}$, and since it is
orthogonal to the limit of $\sum_{k\ne j,j-1}\mathcal Y_k$
with respect to the symplectic inner product $\Omega$,
it therefore follows that $Z_{j,s_1}(s_2)=Z_{j-1,s_1}(s_2)=0$. This finishes the proof of the proposition.
\end{proof}
\begin{remark}\label{remark:jac}
For $i$ with $b_i\ne b_{i-1}$ and $b_i\ne 
b_{i+1}$, Propositions \ref{prop:jf1}, \ref{prop:jacreg} 
and Corollary \ref{cor:conj} equally hold for
the Jacobi field $Y_{i,s}(t)$.
\end{remark}

\section{Geodesics starting at a one point}
In this and the subsequent sections we shall assume
that the condition (\ref{cond2}) are satisfied.
Let $p_0\in M$ be an arbitrary point. We may assume without
loss of generality that $p_0$ is
represented by $(x_1,\dots,x_n)=
(x_1^0,\dots,x_n^0)$, where $0\le x_i^0\le\alpha_i/4$ $(1\le i\le n)$.
Let $U^*_{p_0}M$ be the sphere of unit covectors at $p_0$.
We denote by 
\begin{equation*}
t\mapsto\gamma(t,\eta)=
(x_1(t,\eta),\dots,x_n(t,\eta))
\end{equation*}
the geodesic with the initial covector 
$\eta\in U^*_{p_0}M$ at $t=0$.  
The function $x_i(t,\eta)$ is uniquely determined as a smooth 
function when $b_i\ne a_i$ and $b_{i-1}\ne a_{i-1}$
for each $i$. In this case,
the geodesic does not meet $J_i\cup J_{i-1}$, a part of the branch locus.
If $b_i=a_i$, then the geodesic meets 
$J_i$ and one gets more
than one representations for $x_i(t,\eta)$ and 
$x_{i+1}(t,\eta)$ that are
continuous at the branch point and smooth elsewhere. Note
that $t\mapsto f_i(x_i(t,\eta))$ is uniquely determined in any
case.

As before, we put
\begin{equation*}
\sigma_i(t,\eta)=\int_0^t\left|
\frac{df_i(x_i(t,\eta))}{dt}\right|\,dt\ .
\end{equation*}
We shall assign a real number $t_0(\eta)>0$ to each 
$\eta\in U^*_{p_0}M$. First we consider the case which is {\it not}
equal to  any one of the following three cases: (i) the geodesic 
$\gamma(t,\eta)$ is totally contained in the submanifold
$N_n$, i.e., $b_{n-1}= a_n$; (ii)  $\gamma(t,\eta)$ is 
totally contained in the submanifold $N_{n-1}$ and 
$f_n(x_n^0)=a_{n-1}=b_{n-1}<f_{n-1}(x_{n-1}^0)$; and (iii)
$\gamma(t,\eta)$ is 
totally contained in the submanifold $N_{n-1}$ and $p_0\in J_{n-1}$, in particular, $f_n(x_n^0)=a_{n-1}=b_{n-1}=f_{n-1}(x_{n-1}^0)$. Then, 
define $t_0(\eta)$ by the formula
\begin{equation*}
\sigma_n(t_0(\eta),\eta)=2(a_{n-1}^--a_n^+)\ .
\end{equation*}

In the cases (i) and (ii) listed above, we define 
$t_0(\eta)$ as follows:  Let $Y(t)$ 
be the
Jacobi field along the geodesic $\gamma(t,\eta)$ such that
$Y(0)=0$ and $Y'(0)=(\partial/\partial x_n)/
|\partial/\partial x_n|$. Then $t=t_0(\eta)$
is the first positive time such that $Y(t)=0$.
In the case (iii) we define the Jacobi field
$Y(t)$ along the geodesic $\gamma(t,\eta)$ such that
$Y(0)=0$ and $Y'(0)$ is the unit normal vector to $N_{n-1}$. Then $t=t_0(\eta)$ is the first positive time such that
$Y(t)=0$. It is easily seen that $x_n(t_0(\eta),\eta)=
-x_n^0$, or $\frac{\alpha_n}2+x_n^0$ in any case.

It will be proved in Theorem \ref{thm:cut} that 
the time $t=t_0(\eta)$ gives the cut point of
$p_0$ along the geodesic $\gamma(t,\eta)$. In particular,
it will become clear that $t_0(\eta)$ is a continuous function 
of $\eta\in U^*_{p_0}M$ and $p_0\in M$. In this stage,
we shall only prove a partial result.
\begin{prop}\label{prop:conti}
For any $\eta\in U^*_{p_0}M$ and $p_0\in M$, there is a
sequence $\eta_k$ $(k=1,2,\dots)$ of unit covectors such
that the corresponding values $b_1,\dots,b_{n-1}$
of $H_1,\dots,$ $H_{n-1}$ at 
$\eta_k$ and $a_0,\dots,a_n$ are all distinct for each $k$,
and 
\begin{equation*}
\lim_{k\to\infty}\eta_k=\eta,\qquad \lim_{k\to\infty}
t_0(\eta_k)=t_0(\eta)\ .
\end{equation*}
\end{prop}
\begin{proof}  At each covector $\eta$ which is not of the
cases (i), (ii), (iii), the function $t_0(\eta)$ is clearly
continuous, and we can find such $\{\eta_k\}$. For $\eta$
of the cases (i) or (ii) we note that $t_0(\eta)$ is equal to
the limit
$\lim_{s\to 0}t_0(\eta_s)$, where $\eta_s\in U^*_{p_0}$ is a
one-parameter family of covectors such that (i)
$b_{n-1}=a_n+s^2$, (ii) $b_{n-1}=a_{n-1}+s^2$, and other $b_j$'s are the same value as those
for $\eta=\eta_0$.

Now, for $\eta\in U^*_{p_0}M$ of the cases (ii), (iii),
we first choose $\{\tilde \eta_k\}\in U^*_{p_k}M$ such that
each $\tilde \eta_k$ is of the case (ii), $\tilde\eta_k\to\eta$ $(k\to\infty)$, and the values $b_1,\dots,b_{n-2}$ 
for each $\tilde\eta_k$ and $a_0,\dots,a_n$ are all distinct.
Then, for each $k$ we choose $\eta_k\in U^*_{p_k}M$ in the one-parameter family of covectors given above whose limit
is $\tilde\eta_k$ so that $\eta_k\to\eta$ as $k\to\infty$.
The case (i) is similar.
\end{proof}
For a while, we shall assume that $p_0\not\in J_{n-1}$. 
Put
\begin{align*}
U_+=&\{\eta\in U^*_{p_0}M\ |\ \xi_n(\eta)>0\}\\
U_-=&\{\eta\in U^*_{p_0}M\ |\ \xi_n(\eta)<0\}\ .
\end{align*}
Note that they are well-defined hemispheres under the
assumption $p_0\not\in J_{n-1}$.
Let $\eta'\in U^*_{p_0}M$
be the reflection image of $\eta\in U^*_{p_0}M$ 
with respect to the hyperplane $H_n$ in
$T^*_{p_0}M$ defined by $\xi_n=0$, i.e.,
$\xi_n(\eta')=-\xi_n(\eta)$, $\xi_i(\eta')=
\xi_i(\eta)$ $(1\le i \le n-1)$. 
\begin{prop}\label{prop:refl}
$\gamma(t_0(\eta'),\eta')=\gamma(t_0(\eta),\eta)$ for any $\eta\in U_+$. 
\end{prop}
\begin{proof}
It is enough to show this for covectors $\eta$ such that 
$b_i$'s and $a_j$'s are all distinct.
By (\ref{eq:geod}) we have
\begin{equation*}
\sum_{i=1}^n\int_0^{t_0(\eta)}\frac{(-1)^iG(f_i)A(f_i)}
{\sqrt{-\prod_{k=1}^{n-1}(f_i-b_k)
\cdot\prod_{k=0}^n(f_i-a_k)}}\ 
\left|\frac{df_i(x_i(t,\eta))}{dt}\right|
\ dt=0
\end{equation*}
for any polynomial $G(\lambda)$ of degree $\le n-2$. 
By using the variables $\sigma_i$ given above, this formula
is rewritten as
\begin{equation}\label{eq:geodsigma}
\sum_{i=1}^n\int_0^{\sigma_i(t_0(\eta),\eta)}\frac{(-1)^iG(f_i)A(f_i)}
{\sqrt{-\prod_{k=1}^{n-1}(f_i-b_k)
\cdot\prod_{k=0}^n(f_i-a_k)}}\ 
\ d\sigma_i=0\ .
\end{equation}
Note that
\begin{equation}\label{eq:geodnth}
\begin{gathered}
\int_0^{\sigma_n(t_0(\eta),\eta)}\frac{(-1)^iG(f_n)A(f_n)}
{\sqrt{-\prod_{k=1}^{n-1}(f_n-b_k)
\cdot\prod_{k=0}^n(f_n-a_k)}}\ 
\ d\sigma_n\\
=2\int_{a_n^+}^{a_{n-1}^-}\frac{(-1)^iG(\lambda)A(\lambda)}
{\sqrt{-\prod_{k=1}^{n-1}(\lambda-b_k)
\cdot\prod_{k=0}^n(\lambda-a_k)}}
\ d\lambda\ .
\end{gathered}
\end{equation}
Since the values of each $b_i$ are the same for the two
covectors $\eta$ and $\eta'$, and since $\sigma_n(t_0(\eta),\eta)=2(a_{n-1}^--a_n^+)=\sigma_n(t_0(\eta'),\eta')$, we then have
\begin{equation}\label{eq:prime2}
\begin{aligned}
\sum_{i=1}^{n-1}\int_0^{\sigma_i(t_0(\eta),\eta)}\frac{(-1)^iG(f_i)A(f_i)}
{\sqrt{-\prod_{k=1}^{n-1}(f_i-b_k)
\cdot\prod_{k=0}^n(f_i-a_k)}}\ 
\ d\sigma_i\\
=\sum_{i=1}^{n-1}\int_0^{\sigma_i(t_0(\eta'),\eta')}\frac{(-1)^iG(f_i)A(f_i)}
{\sqrt{-\prod_{k=1}^{n-1}(f_i-b_k)
\cdot\prod_{k=0}^n(f_i-a_k)}}\ 
\ d\sigma_i
\end{aligned}
\end{equation}

Now, let $I$ be the set of $i\in\{1,\dots,n-1\}$ such that
\begin{equation*}
\sigma_i(t_0(\eta),\eta)>\sigma_i(t_0(\eta'),\eta')\ .
\end{equation*}
Then, as we shall prove in the next lemma, there is a
polynomial $G(\lambda)$ of degree $\le n-2$ such that
$(-1)^iG(\lambda)>0$ for $\lambda\in (a_i^+,a_{i-1}^-)$,
$i\in I$, and $(-1)^iG(\lambda)<0$ for $\lambda\in
(a_i^+, a_{i-1}^-)$, $i\not\in I$, if $I\ne \emptyset$. 
With such
$G(\lambda)$, the formula (\ref{eq:prime2}) clearly yields
a contradiction.  Therefore, $I=\emptyset$ and
\begin{equation*}
\sigma_i(t_0(\eta),\eta)=\sigma_i(t_0(\eta'),\eta')\ .
\end{equation*}
for every $1\le i\le n-1$. This indicates
\begin{equation*}
x_i(t_0(\eta),\eta)=x_i(t_0(\eta'),\eta')\ .
\end{equation*}
for any $1\le i\le n$, and therefore
$\gamma(t_0(\eta'),\eta')=\gamma(t_0(\eta),\eta)$\ .
\end{proof}
\begin{lemma}\label{lem:G}
Suppose $b_i$'s and $a_i$'s are all distinct. Let $I_1$ be a
subset of $\{1,\dots,n\}$ and let $I_2$ be its complement.
Assume both $I_1$ and $I_2$ are nonempty. Then there is a
polynomial $G(\lambda)$ of degree $\le n-2$ such that
\begin{equation*}
(-1)^iG(\lambda)
\begin{cases}
>0\quad \text{for } \lambda\in (a_i^+,a_{i-1}^-),\ i\in I_1\\
<0\quad \text{for } \lambda\in (a_i^+,a_{i-1}^-),\ i\in I_2
\end{cases}\ .
\end{equation*}
\end{lemma}
\begin{proof}
Assume $1\in I_1$. We put 
\begin{equation*}
G(\lambda)=-\prod(\lambda-b_k)\ ,
\end{equation*}
where the product are taken over all such 
$k\in\{1,\dots,n-1\}$ that both $k$ and $k+1$ belongs to $I_1$ or that both $k$ and $k+1$ belongs to $I_2$. Since both $I_1$
and $I_2$ are nonempty, it follows that $\deg G\le n-2$.
Also, it is clear that the signs of the function $G(\lambda)$
is different on the two intervals $(a_k^+,a_{k-1}^-)$ and
$(a_{k+1}^+,a_k^-)$ if and only if $\lambda-b_k$ is a factor
of $G(\lambda)$, i.e., $k$ and $k+1$ belong to the same 
group.  Since $-G(\lambda)>0$ on $(a_1^+,a_0^-)$,
it follows that this $G(\lambda)$ has the desired property.
In case $1\in I_2$, then $-G(\lambda)$ possesses the desired
property.
\end{proof}

\begin{prop}\label{prop:prime}
$t_0(\eta)=t_0(\eta')$ for any $\eta\in U^*_{p_0}M$.
\end{prop}
\begin{proof}
By (\ref{eq:length}) we have
\begin{equation}\label{eq:t0}
t_0(\eta)=\sum_{i=1}^n\int_0^{\sigma_i(t_0(\eta),\eta)}\frac{(-1)^{i+1}A(f_i)\prod_{k=1}^{n-1}(f_i-a_k)}
{2\sqrt{-\prod_{k=1}^{n-1}(f_i-b_k)
\cdot\prod_{k=0}^n(f_i-a_k)}}\ 
\ d\sigma_i
\end{equation}
Since $\sigma_i(t_0(\eta),\eta)=\sigma_i(t_0(\eta'),\eta')$
for any $i$ by Proposition \ref{prop:refl}, it  therefore
follows that $t_0(\eta)=t_0(\eta')$.
\end{proof}
\begin{prop}\label{prop:sigma}
Suppose that the geodesic $\gamma(t,\eta)$ does not totally
contained in any $N_j$ for any $j$. Then,
$\sigma_i(t_0(\eta),\eta)< 2(a_{i-1}^--a_i^+)$ for 
any $i\le n-1$ such that $b_i\ne b_{i-1}$.
\end{prop}
\begin{proof}
The assumption implies that there is no $i$
such that $b_i=a_{i+1}$ or $b_{i+1}=a_i$.
First, suppose that $b_1,\dots,b_{n-1}$ and $a_0,
\dots,a_n$ are all distinct. Let $I_1$ be the set of 
$i\in\{1,\dots,n-1\}$ such that $\sigma_i(t_0(\eta),\eta)
\ge 2(a_{i-1}^--a_i^+)$. Assume that $I_1\ne\emptyset$. 
Put $I_2=\{1,\dots,n\}-I_1$. Note that $n\in I_2$. For these
$I_1$ and $I_2$, let $G(\lambda)$ be the polynomial given in
the proof of Lemma \ref{lem:G}. Then we have
\begin{equation}\label{eq:ineq}
\begin{aligned}
&2\sum_{i=1}^n\int_{a_i^+}^{a_{i-1}^-}
\frac{(-1)^{i}G(\lambda)A(\lambda)\ d\lambda}
{\sqrt{-\prod_{k=1}^{n-1}(\lambda-b_k)
\cdot\prod_{k=0}^n(\lambda-a_k)}}\\
=&-\sum_{i\in I_1}\int_{2(a_{i-1}^--a_i^+)}^{\sigma_i(t_0(\eta),\eta)}\frac{(-1)^iG(f_i)A(f_i)}
{\sqrt{-\prod_{k=1}^{n-1}(f_i-b_k)
\cdot\prod_{k=0}^n(f_i-a_k)}}\ 
\ d\sigma_i\\
+&\sum_{i\in I_2-\{n\}}\int_{\sigma_i(t_0(\eta),\eta)}^{2(a_{i-1}^--a_i^+)}\frac{(-1)^iG(f_i)A(f_i)}
{\sqrt{-\prod_{k=1}^{n-1}(f_i-b_k)
\cdot\prod_{k=0}^n(f_i-a_k)}}\ 
\ d\sigma_i\ .
\end{aligned}
\end{equation}
Here, the polynomial $G(\lambda)$ is of the form
\begin{equation*}
G(\lambda)=
\begin{cases}
-\prod_{k\in K}(\lambda-b_k)\quad (\text{if }1\in I_1)\\
\prod_{k\in K}(\lambda-b_k)\quad (\text{if }1\in I_2)
\end{cases}\ ,
\end{equation*}
where $K$ is the subset of $\{1,\dots,n-1\}$ such that
$k\in  K$ means $k$ and $k+1$ belong to the same
group, i.e., $k,k+1\in I_1$, or $k,k+1\in I_2$.
Therefore, $n-1-\#K$ is the number of such $k\in\{1,\dots,n-1\}$ that $k$ and
$k+1$ belong to the different groups. Since $n\in I_2$,
it follows that
\begin{equation*}
n-1-\#K\ \text{is }
\begin{cases}
\text{ odd\quad if }\ 1\in I_1\\
\text{ even\quad if }\ 1\in I_2.
\end{cases}
\end{equation*}
Therefore, by Proposition \ref{prop:cond2} (1)
it follows that the
first line in the formulas (\ref{eq:ineq}) is
positive, while the second and the third lines are
nonpositive, which is a contradiction. Thus $I_1$ must be
empty, and the proposition follows.

Next, we shall consider the case where $b_{j-1}=b_j$
for several $j$, but other $b_k$ and $a_k$ are all
distinct. In this case, we define the subset $I_1$
of $\{1,\dots,n-1\}$ as follows: For $k$ with $b_{k-1}\ne b_k$, $k\in I_1$ if and only if $\sigma_k
(t_0(\eta),\eta)\ge 2(a_{k-1}^--a_k^+)$; for $k$
with $b_{k-1}=b_k$, $k\in I_1$ if and only if $k-1
\in I_1$ or $k+1\in I_1$. Note that $b_{k-1}<b_{k-2}$
and $b_{k+1}<b_k$ if $b_k=b_{k-1}$.

Then, by the same way as above, we define the sets
$I_2$, $K$ and the polynomial $G(\lambda)$. Put
\begin{equation*}
J=\{j\ |\ b_j<b_{j-1}, 1\le j\le n-1\}\ .
\end{equation*}
Since $k-1\in K$ or $k\in K$ if $b_k=b_{k-1}$,
we then have, instead of \eqref{eq:ineq}, the
following formula:
\begin{equation}\label{eq:ineq2}
\begin{aligned}
&2\sum_{i\in J}\int_{a_i^+}^{a_{i-1}^-}
\frac{(-1)^{i}G(\lambda)A(\lambda)\ d\lambda}
{\sqrt{-\prod_{k=1}^{n-1}(\lambda-b_k)
\cdot\prod_{k=0}^n(\lambda-a_k)}}\\
=&-\sum_{i\in I_1\cap J}\int_{2(a_{i-1}^--a_i^+)}^{\sigma_i(t_0(\eta),\eta)}\frac{(-1)^iG(f_i)A(f_i)}
{\sqrt{-\prod_{k=1}^{n-1}(f_i-b_k)
\cdot\prod_{k=0}^n(f_i-a_k)}}\ 
\ d\sigma_i\\
+&\sum_{i\in I_2\cap J}\int_{\sigma_i(t_0(\eta),\eta)}^{2(a_{i-1}^--a_i^+)}\frac{(-1)^iG(f_i)A(f_i)}
{\sqrt{-\prod_{k=1}^{n-1}(f_i-b_k)
\cdot\prod_{k=0}^n(f_i-a_k)}}\ 
\ d\sigma_i\ .
\end{aligned}
\end{equation}
If $I_1\cap J\ne\emptyset$, then we have a 
contradiction by the same reason as above.

Finally, let us further assume that $b_i=a_i$
for some $i$. In this case,
the times t such that $f_i(x_i(t,\eta))=a_i$ and those such
that $f_{i+1}(x_{i+1}(t,\eta))=a_i$ coincide. Therefore,
in each side of the formula \eqref{eq:ineq} or
\eqref{eq:ineq2}, the sum of the
integrals in $\sigma_i$ and $\sigma_{i+1}$ remains finite,
and the arguments above are also effective in this case.
\end{proof}
\begin{prop}\label{prop:ineqsing}
Suppose that the geodesic $\gamma(t,\eta)$ does not totally
contained in any $N_k$. 
For a fixed $j$ with $b_j=b_{j-1}$, let $\theta_{s_1}(s_2)$
be the value defined in the formula
\eqref{eq:theta} in the previous section.
Then, $\theta_0(t_0(\eta))< \pi$
for such $j$.
\end{prop}
\begin{proof}
By \eqref{eq:theta} we have
\begin{equation*}
\begin{gathered}
\sum_{\substack{1\le l\le n\\l\ne j}}\int_{0}^{\sigma_l(s)}\frac{(-1)^l
G_{j,j-1}(f_l)A(f_l)\ d\sigma_l}
{|f_l-\lambda_j^0|\sqrt{-\prod_{k\ne j,j-1}(f_l-b_k)
\cdot\prod_{k=0}^n(f_l-a_k)}}\\
+2\theta_{0}(s)\ \frac{(-1)^jG_{j,j-1}(\lambda_j^0)A(\lambda_j^0)}
{\sqrt{\prod_{k\ne j,j-1}(\lambda_j^0-b_k)
\prod_k(\lambda_j^0-a_k)}}\ =0\ .
\end{gathered}
\end{equation*}
Also, taking a limit $a_j^+,a_{j-1}^-\to\lambda_j^0$ in Lemma \ref{lemma:flat}, we have
\begin{equation*}
\begin{gathered}
\sum_{\substack{1\le l\le n\\l\ne j}}\int_{0}^{2(a_{l-1}^--
a_l^+)}\frac{(-1)^l
G_{j,j-1}(f_l)A(\lambda_j^0)\ d\sigma_l}
{|f_l-\lambda_j^0|\sqrt{-\prod_{k\ne j,j-1}(f_l-b_k)
\cdot\prod_{k=0}^n(f_l-a_k)}}\\
+2\pi\ \frac{(-1)^jG_{j,j-1}(\lambda_j^0)A(\lambda_j^0)}
{\sqrt{\prod_{k\ne j,j-1}(\lambda_j^0-b_k)
\prod_k(\lambda_j^0-a_k)}}\ =0\ .
\end{gathered}
\end{equation*}
Therefore we obtain the following formula:
\begin{equation*}
\begin{gathered}
\sum_{\substack{1\le l\le n\\l\ne j}}\int_{\sigma_l(s)}^{2
(a_{l-1}^--a_l^+)}\frac{(-1)^l
G_{j,j-1}(f_l)A(f_l)\ d\sigma_l}
{|f_l-\lambda_j^0|\sqrt{-\prod_{k\ne j,j-1}(f_l-b_k)
\cdot\prod_{k=0}^n(f_l-a_k)}}\\
-\sum_{\substack{1\le l\le n\\l\ne j}}\int_{0}^{2
(a_{l-1}^--a_l^+)}\frac{(A(f_l)-A(\lambda_j^0))\,(-1)^l
G_{j,j-1}(f_l)\ d\sigma_l}
{|f_l-\lambda_j^0|\sqrt{-\prod_{k\ne j,j-1}(f_l-b_k)
\cdot\prod_{k=0}^n(f_l-a_k)}}\\
+2(\pi-\theta_{0}(s))\ \frac{(-1)^jG_{j,j-1}(\lambda_j^0)A(\lambda_j^0)}
{\sqrt{\prod_{k\ne j,j-1}(\lambda_j^0-b_k)
\prod_k(\lambda_j^0-a_k)}}\ =0\ .
\end{gathered}
\end{equation*}
We put $s=t_0(\eta)$.
The first line of this formula is nonpositive by the 
previous proposition. Also, applying the $n-1$-dimensional
version of Proposition \ref{prop:cond2} (1) to the positive 
function 
\begin{equation*}
\left(A(\lambda)-A(\lambda_j^0)\right)/(\lambda-\lambda_j^0)
\ ,
\end{equation*}
the second line is negative. Since $(-1)^jG_{j,j-1}(\lambda_j^0)>0$, 
it thus follows that $\theta_0(t_0(\eta))<
\pi$.
\end{proof}

As a consequence, we have the following proposition.
\begin{prop}\label{prop:noconj}
Suppose that the geodesic $\gamma(t,\eta)$ does not totally
contained in any $N_k$. Then:
\begin{enumerate}
\item There is no conjugate point of $p_0$ along the geodesic
$\gamma(t,\eta)$ in the interval $0<t<t_0(\eta)$.
\item $\gamma(t_0(\eta),\eta)$ is not a conjugate point of
$p_0$ along the geodesic $\gamma(t,\eta)$, unless $b_{n-1}(=H_{n-1}(\eta))=f_n(x_n^0)$.
\item If $b_{n-1}=f_n(x_n^0)$, then $\gamma(t_0(\eta),\eta)$ 
is a conjugate point of $p_0$ along the geodesic 
$\gamma(t,\eta)$ with multiplicity one.
\end{enumerate}
\end{prop}
\begin{proof}
(1) and (2) follow from all results in \S4 and
Propositions \ref{prop:sigma} and 
\ref{prop:ineqsing}. Now, let us prove (3).
Since $f_n(x_n^0)=b_{n-1}$, it follows 
from Corollary \ref{cor:conj} (1) that 
$Y_{n-1,0}(t_0(\eta))=0$. Hence $\gamma(t_0(\eta),\eta)$ 
is a conjugate point of $p_0$ along the geodesic 
$\gamma(t,\eta)$. Now we show that $Y_{j,0}
(t_0(\eta))\ne 0$ (or, $Z_{j,0}(t_0(\eta))\ne 0$)
for any $j\le n-2$. First, suppose that $b_j\ne 
b_{j-1}$ for any $j$. For $k\le n-2$ with
$b_k\ne f_k(x_k^0)$, $f_{k+1}(x_{k+1}^0)$ , we
have $Y_{k,0}(t_0(\eta))\ne 0$ by Propositions 
\ref{prop:sigma} and \ref{prop:jacreg}. If
$b_k= f_k(x_k^0)$ or $f_{k+1}(x_{k+1}^0)$,
then again we have $Y_{k,0}(t_0(\eta))\ne 0$ by
Proposition \ref{prop:sigma} and Corollary
\ref{cor:conj} (1). In case $b_j=b_{j-1}$ for 
some $j$, we also have $Z_{j,0}(t_0(\eta))\ne 0$
and $Z_{j-1,0}(t_0(\eta))\ne 0$ in the same way as above by Proposition \ref{prop:ineqsing}.
\end{proof}
\section{Cut locus (1)}
Let $p_0$ be a point as in \S5.
Let $N$ be the subset of $M$ represented by
$x_n=\frac{\alpha_n}2+x_n^0$ or $-x_n^0$,
which is a submanifold of $M$ diffeomorphic to
the $(n-1)$-sphere if $0\le x_n^0< \alpha_n/4$, and which is a 
submanifold with boundary diffeomorphic to closed $(n-1)$-disk 
if $x_n^0=\alpha_n/4$. Let $t_0(\eta)$ be the value defined
in the previous section.
\begin{thm}\label{thm:cut}
\begin{enumerate}
\item  The cut point of $p_0$ along
the geodesic $\gamma(t,\eta)$ is given by $t=t_0(\eta)$
for any $p_0\in M$ and $\eta\in U^*_{p_0}M$.
\item Suppose $p_0\not\in J_{n-1}$. Then, the assignment
$\eta\mapsto \gamma(t_0(\eta),\eta)$ gives
a homeomorphism from $\overline{U_+}$ to its image
$C(p_0)$, the cut locus of $p_0$, and it gives $C^\infty$ 
embeddings of $U_+$ and $\partial \overline{U_+}$ 
respectively.
In particular, $C(p_0)$ is diffeomorphic to an $(n-1)$-closed 
disk, and it is contained in (the interior of) $N$. Also,
for each $\eta\in \partial \overline{U_+}$, $\gamma(t_0(\eta),\eta)$ is the
first conjugate point of $p_0$ of multiplicity one along
the geodesic $t\mapsto\gamma(t,\eta)$ .
\item Suppose $p_0\in J_{n-1}$. Then the cut locus $C(p_0)$ coincides with the cut locus of
$p_0$ in the totally geodesic submanifold $N_{n-1}$, which is smoothly embedded $(n-2)$-disk
in $J_{n-1}$. For each interior 
point $q$ of $C(p_0)$ there is an $S^1$-family of minimal 
geodesics joining $p_0$ and $q$; the tangent vectors
of those geodesics at $p_0$ form a cone whose orthogonal 
projection to $T_{p_0}J_{n-1}$ is one-dimensional. For each boundary point $q$ of 
$C(p_0)$, there is a unique minimal geodesic from $p_0$ to $q$, 
and along it $q$ is the first conjugate point of $p_0$ of
multiplicity two.
\end{enumerate}
\end{thm}
In this and the next two sections, we shall prove this theorem.
The proof will be divided into five cases: (I) 
$p_0\not\in N_k$ for any $k$; (II) $0<x_n^0<\alpha_n/4$, but $p_0\in N_l$ for some $l$; (III) $x_n^0=0$; (IV) $x_n^0=\alpha_n/4$, and $p_0\not\in 
J_{n-1}$; (V) $p_0\in J_{n-1}$.
In this section we shall consider the case (I) and prove (1) and (2)
of the theorem in this case. The proofs for the cases (II) $\sim$ (V) will be given in the next
two sections.

For each $\eta\in U_-$, let $t_-(\eta)$ be the
first positive time $t$ such that $x_n(t, \eta)
=-x_n^0$. Define the mapping $\Phi:U_{p_0}^*M\to N$ by
\begin{equation*}
\Phi(\eta)=\gamma(t_0(\eta),\eta)\quad (\eta\in \overline{
U^+});\qquad =\gamma(t_-(\eta),\eta)\quad (\eta\in \overline{
U_-})\ .
\end{equation*}
Then, $\Phi(\eta)\in N$ is the first point that the geodesic
$\gamma(t,\eta)$ meets $N$ for any $\eta$. We shall prove
that $\Phi$ is a homeomorphism. To do so, we need several
lemmas.

Take a point $p'_0$ in such
a way that $p'_0$ is represented as
$(x_1^0,\dots,x_{n-1}^0,\allowbreak x_n^1)$, where
$0\le x_n^1<x_n^0< \alpha_n/4$. Let $U_+'$ be the hemisphere of $U^*_{p'_0}M$ defined by
$\xi_n>0$. We define the mapping 
$\psi: \overline{U_+}\to U'_+$
so that it preserves the values $b_i$ of $H_i$
$(1\le i\le n-1)$, i.e., by 
$\psi(p_0; \xi_1,\dots,\xi_n)
=(p'_0; \tilde\xi_1,\dots,\tilde\xi_n)$, where
\begin{equation*}
\tilde\xi_i=\xi_i\quad (1\le i\le n-1),\qquad
\tilde\xi_n=\sqrt{(-1)^{n-1}\prod_{k=1}^{n-1}
(f_n(x_n^1)-b_k)}\ .
\end{equation*}
Note that $b_k$'s are functions of
$(p_0; \xi_1,\dots,\xi_n)\in \overline{U_+}$.
Since $b_{n-1}\ge f_n(x_n^0)>f_n(x_n^1)$, 
the image $\psi(\overline{U_+})$ is contained in
the interior $U'_+$. Let $N'$ be the
submanifold of $M$ defined by 
$x_n=-x_n^1$, and define the diffeomorphism
$\Psi: N\to N'$ by
\begin{equation*}
\Psi(x_1,\dots,x_{n-1},-x_n^0)=
(x_1,\dots,x_{n-1},-x_n^1).
\end{equation*}
We also define $\tilde\Phi:U_+'\to N'$
in the same way as $\Phi|_{\overline{U_+}}$.
\begin{lemma}\label{lemma:shift}
$\Psi(\Phi(\eta))=\tilde\Phi(\psi(\eta))$
for any $\eta\in\overline{U_+}$.
\end{lemma}
\begin{proof}
We write $\psi(\eta)=\tilde \eta$ for simplicity.
For the geodesics $\gamma(t,\eta)$ and $\gamma(t,\tilde\eta)$, we have the equality (\ref{eq:geodsigma}) and the similar one.
Taking the equality (\ref{eq:geodnth}) into
account, we have the similar formula as 
(\ref{eq:prime2}):
\begin{gather*}
\sum_{i=1}^{n-1}\int_0^{\sigma_i(t_0(\eta),\eta)}\frac{(-1)^iG(f_i)A(f_i)}
{\sqrt{-\prod_{k=1}^{n-1}(f_i-b_k)
\cdot\prod_{k=0}^n(f_i-a_k)}}\ 
\ d\sigma_i\\
=\sum_{i=1}^{n-1}\int_0^{\sigma_i(t_0(\tilde\eta),\tilde\eta)}\frac{(-1)^iG(f_i)A(f_i)}
{\sqrt{-\prod_{k=1}^{n-1}(f_i-b_k)
\cdot\prod_{k=0}^n(f_i-a_k)}}\ 
\ d\sigma_i\ .
\end{gather*}
Therefore, in the same way as the proof of
Proposition \ref{prop:refl}, we have
$\sigma_i(t_0(\tilde\eta),\tilde\eta)=
\sigma_i(t_0(\eta),\eta)$ and hence
$x_i(t_0(\tilde\eta),\tilde\eta)=
x_i(t_0(\eta),\eta)$ for any $i\le n-1$.
Thus we have $\gamma(t_0(\tilde\eta),\tilde\eta)=
\Psi(\gamma(t_0(\eta),\eta))$. By the formula
(\ref{eq:t0}) we also have $t_0(\tilde\eta)=
t_0(\eta)$.
\end{proof}
By Proposition \ref{prop:noconj}, we know that
$\Phi|_{U_+}$ is a local diffeomorphism and so
is true for the initial point $p'_0$. 
Therefore it follows from the above lemma that
$\Phi|_{\overline{U_+}}$ is a local 
homeomorphism and $\Phi|_{\partial \overline{U_+}}$ is a local diffeomorphism. For the mapping $\Phi$ on 
$\overline{U_-}$, we have the following
\begin{lemma}
$\Phi|_{\overline{U_-}}$ is a $C^1$ local
diffeomorphism.
\end{lemma}
\begin{proof}
By Proposition \ref{prop:noconj} and by the above 
observation, we know that
$\Phi|_{U_-}$ and $\Phi|_{\partial\overline{U_-}}$ $(=\Phi|_{\partial\overline{U_+}})$ are $C^\infty$ immersions. 
Let $\{\eta_s\}$
be a one-parameter family of unit covectors at
$p_0$ such that $\eta_s\in U_-$ $(s>0)$, $\eta_0
\in \partial\overline{U_-}$, and $\dot{\eta}_s
=\left(\partial/\partial \nu_{n-1}\right)/
|\partial/\partial \nu_{n-1}|$, where the variable
$\nu_{n-1}$ is the one defined in 
\S\ref{sec:jac}. 
We shall show that $\Phi|_{\overline{U_-}}$ is of class $C^1$ and a local diffeomorphism at $\eta_0$.

Differentiating the equality
\begin{equation*}
\sum_{l=1}^n\int_0^{\sigma_l(t_-(\eta_s),\eta_s)}\frac{(-1)^{l}G_{n-1}(f_l)A(f_l)\ d\sigma_l}
{\sqrt{-\prod_{k=1}^{n-1}(f_l-b_k)
\cdot\prod_{k=0}^n(f_l-a_k)}}=0
\end{equation*}
in $s$, one obtains
\begin{equation}\label{eq:t1}
\begin{gathered}
0=\beta(c_s\,Y_{n-1,0}(t_-(\eta_s))+\frac{\partial}{\partial s}
t_-(\eta_s)\cdot \dot\gamma(t_-(\eta_s),\eta_s))\\
-\nu_{n-1}\sum_{l=1}^n\int_{0}^{\sigma_l(t_-(\eta_s),\eta_s)}\frac{(-1)^iG_{n-1}(f_l)A(f_l)\ d\sigma_l}
{(f_l-b_{n-1})\sqrt{-\prod_{k=1}^{n-1}(f_l-b_k)
\cdot\prod_{k=0}^n(f_l-a_k)}}\ ,
\end{gathered}
\end{equation}
where $c_s=\pm |\partial/\partial \nu_{n-1}|$ at
$\eta_s$ and $\beta$ is the 1-form;
\begin{equation*}
\beta=\sum_{l=1}^{n-1}\frac{\epsilon'_l(-1)^lG_{n-1}(f_l(x_l))A(f_l(x_l))}
{\sqrt{-\prod_{k=1}^{n-1}(f_l(x_l)-b_k)
\cdot\prod_{k=0}^n(f_l(x_l)-a_k)}}\ d(f_l(x_l))\ .
\end{equation*}
Then, taking the limit $s\searrow 0$, we have
\begin{equation*}
\begin{gathered}
0=\frac{\partial}{\partial s}
t_-(\eta_s)\big|_{s=0}\ \beta(\dot\gamma(t_-(\eta_0),\eta_0))\\
+\ \frac{4\epsilon'_{n}(-1)^{n}G_{n-1}(b_{n-1})A(b_{n-1})}
{\sqrt{-\prod_{k\ne n-1}(b_{n-1}-b_k)
\cdot\prod_{k=0}^n(b_{n-1}-a_k)}}\ .
\end{gathered}
\end{equation*}
Noting that the covector $\flat(\dot\gamma(t_-(\eta_0),\eta_0))$
is equal to 
\begin{equation*}
\frac12\sum_{l=1}^{n-1}\frac{\epsilon'_l(-1)^{l+1}A(f_l(x_l))
\prod_{k=1}^{n-1}(f_l(x_l)-b_k)}
{\sqrt{-\prod_{k=1}^{n-1}(f_l(x_l)-b_k)
\cdot\prod_{k=0}^n(f_l(x_l)-a_k)}}\ d(f_l(x_l))
\end{equation*}
at $\gamma(t_-(\eta_0),\eta_0)$, we see that
\begin{equation*}
\frac1{b_1-b_{n-1}}<-\beta(\dot\gamma(t_-(\eta_0),\eta_0))<\frac1{b_{n-2}-b_{n-1}}\ .
\end{equation*}
This indicates that $(\partial/\partial s)t_-(\eta_s)|_{s=0}$
is finite and nonzero. 

Also, by similar formulas to \eqref{eq:t1}, the 
derivatives of $\gamma(t_-(\eta),\eta)$
by the normalized $\partial/\partial H_j$ $(j\le 
n-2)$ are of 
the form $Y_{j,0}(t_-(\eta))
+c_\eta\dot\gamma(t_-(\eta)\eta)$ (or 
$Z_{j,0}(t_-(\eta))+c_\eta\dot\gamma(t_-(\eta)\eta)$) 
$\in T_{\gamma(t_-(\eta),\eta)}N$,
which are continuous in $\eta$ near the boundary
$\partial\overline{U_-}$.
Therefore the mapping $\Phi|_{\overline{U_-}}$ is of class $C_1$ and the lemma follows.
\end{proof}
The above lemma implies that $\Phi|_{\overline{U_-}}$ is a
local homeomorphism. Thus, combined with the above result,
we see that $\Phi: U_{p_0}^*M\to N$ is a local
homeomorphism. Since both $U_{p_0}^*M$ and $N$ are homeomorphic to the $(n-1)$-sphere, and since $n\ge 3$, it
therefore follows that $\Phi$ is really a homeomorphism.

We shall prove that the image
of the map $\overline{U_+}\ni\eta\mapsto \gamma(t_0(\eta),\eta)$ is just the cut locus of $p_0$. Let us temporarily denote this image by $\mathcal C$. Note that,
for any $\eta\in U^*_{p_0}M$, the cut point of $p_0$ along
the geodesic $\gamma(t,\eta)$ will appear at
$t\le t_0(\eta)$, because of Propositions \ref{prop:prime}
and \ref{prop:refl}.
In particular, putting 
\begin{equation*}
V=\{t\eta\in T^*_{p_0}M\ |\ \eta\in U^*_{p_0}M,\ 
0\le t<t_0(\eta)\}\ ,
\end{equation*}
we have the following lemma. Put Exp$_{p_0}(t\eta)=\gamma(t,\eta)$.
\begin{lemma}\label{lem}
\begin{enumerate}
\item $\text{Exp}_{p_0}:\overline{V}\to M$ is surjective.
\item $\text{Exp}_{p_0}(V)\cap \mathcal C=\emptyset$.
\end{enumerate}
\end{lemma}
\begin{proof}
Let $q\in M$ be any point $(\ne p_0)$ and let 
$\gamma(t,\eta)$
$(0\le t\le T)$ be a minimal geodesic joining $p_0$ and $q$
$(\eta\in U^*_{p_0}M)$. Since $T\le t_0(\eta)$, (1) follows.
Next, assume that there is some $\eta\in U^*_{p_0}M$ and
$0<T<t_0(\eta)$ such that $\gamma(T,\eta)\in \mathcal C$.
Then, $x_n(T,\eta)=-x_n^0$ or $\frac{\alpha_n}2+x_n^0$.
Note that, if $\eta\in \overline{U_+}$, then $t=t_0(\eta)$ 
is the first positive time when $x_n(T,\eta)=-x_n^0$ or 
$\frac{\alpha_n}2+x_n^0$. Thus we have $\eta\in U_-$ and 
$T=t_-(\eta)$. But, as we have proved in the previous lemma, $\gamma(T,\eta)
\not\in \mathcal C$ in this case, a contradiction. Thus
(2) follows.
\end{proof}

Fix $\eta\in U_{p_0}^*M$ and suppose that
the cut point of $p_0$ along the geodesic $\gamma(t,\eta)$ 
appear before $t=t_0(\eta)$, i.e., the geodesic segment
$\gamma(t,\eta)$ $(0\le t\le t_0(\eta))$ is no longer minimal.
Then there is another minimal geodesic $\gamma(t,\bar\eta)$ 
$(0\le t\le T)$
joining $p_0$ and $q=\gamma(t_0(\eta),\eta)$, $\bar\eta\in U^*_{p_0}M$.

Since the geodesic segment $\gamma(t,\bar\eta)$ 
$(0\le t\le T)$ is minimal, we have $T\le t_0(\bar\eta)$.
Also, since $\gamma(T,\bar\eta)=q\in\mathcal C$,
we have $T=t_0(\bar\eta)$ by Lemma~\ref{lem} (2).
Then, by the injectivity of $\Phi$ we have $\bar\eta=
\eta$ or $\eta'$. But this implies that the geodesic segment
$\gamma(t,\eta)$ $(0\le t\le t_0(\eta))$ is minimal, a
contradiction. Thus $t=t_0(\eta)$ gives the cut point of 
$p_0$ along the geodesic $\gamma(t,\eta)$. This completes the 
proof of (1) and (2) of the theorem
in the case where $0<x_i^0<\alpha_n/4$ for any $i$.

\section{Cut locus (2)}
In this section, we shall give a proof of Theorem \ref{thm:cut} for the case (II) described in the
previous section. The cases (III) $\sim$ (V) will be considered in the next section. 
Note that the statement (1) of the theorem holds for 
any $p_0$ and any $\eta\in U^*_{p_0}M$, which is a 
consequence of the results in the previous section, 
Proposition \ref{prop:conti}, and the continuous dependence 
of cut points on the initial covectors. Thus we shall prove
(2) for the cases (II) $\sim$ (IV) and (3) for the case (V).

Now, let us consider the case (II); $0<x_n^0<\alpha_n/4$ and $p_0\in N_l$ for some
$l\le  n-1$. As in the previous section, we shall show that
$\Phi:U_{p_0}^*M\to N$ is a homeomorphism. 
\begin{prop}\label{prop:noconj2}
Suppose $p_0\in N_l$ and let $\eta\in U_{p_0}^*M$ be a 
covector such that the geodesic $\gamma(t,\eta)$ is totally
contained in $N_l$.  Let $Y_l(t)$ be a nonzero Jacobi field along
the geodesic $\gamma(t,\eta)$ such that $Y_l(0)=0$ and
$Y_l(t)$ is orthogonal to $N_l$ everywhere. Then,
$Y_l(t_0(\eta))\ne 0$.
\end{prop}
The proof will be given below. This proposition
together with Proposition \ref{prop:noconj}
applied to the intersection of the Liouville manifolds $N_l$ in which the geodesic is contained show that the
mapping $\Phi|_{U_+}$ and $\Phi|_{\partial\overline{U_+}}$ are immersions.
Then, in the same way as the previous section,
we see that $\Phi|_{\overline{U_+}}$ is a local
homeomorphism. On the other hand, since $t_0(\eta)$ represents the cut point, and since
$t_-(\eta)<t_0(\eta)$, the mapping $\Phi|_{U_-}$
is a $C^\infty$ embedding and $\Phi(U_-)\cap \Phi(\overline{U_+})=\emptyset$. Also $\Phi
(U^*_{p_0})=N$ by continuity.  Therefore it follows
that $\Phi: U_{p_0}^*M\to N$ is a homeomorphism.
This indicates (2) of the theorem in this case.

In the rest of this section we shall prove 
Proposition \ref{prop:noconj2}.
We may assume that
there is only one such $l$ that the geodesic is
totally contained in $N_l$.
According to the position  of the geodesic
$\gamma(t,\eta)$, there are four different cases: (i) the
geodesic $\gamma(t,\eta)$ intersects $J_l$ transversally;
(ii) $\gamma(t,\eta)$ does not meet $J_l$; (iii)
$\gamma(t,\eta)$ is tangent to $J_l$, but not contained in 
it; (iv) $\gamma(t,\eta)$ is contained in $J_l$.

First, let us consider the case (i), and first
assume $p_0\not\in J_l$. We may also assume
$f_{l+1}(x_{l+1}^0)<b_l=a_l=f_l(x_l^0)$; the case
where $f_{l+1}(x_{l+1}^0)=b_l=a_l<f_l(x_l^0)$ is
similar. Note that $f_l(x_l^0)<b_{l-1}$ in this case,
since the intersection of $\gamma(t,\eta)$
and $J_l$ is transversal in $N_l$. Then the 
Jacobi field $Y_l(t)$ is given by the one-parameter family of geodesics $\{\gamma(t,\eta_s)
\}$, where $\eta_s\in U_{p_0}^*M$ satisfies
$\eta_0=\eta$ and $H_l(\eta_s)=b_l-s^2$, 
$H_j(\eta_s)=b_j$ for $j\ne l$.

To show the proposition in this case, 
we use a similar technique
as Lemma \ref{lemma:shift}, which is as follows.
Take a point $p'_0$ in such
a way that $p'_0$ is represented as
$(x_1^0,\dots,x_{l}^1,\dots,x_n^0)$, where
$0=x_l^0<x_l^1< \alpha_l/4$ and $f_l(x_l^1)<
b_{l-1}, a_{l-1}$. Let $U_{l-}'$ be the hemisphere of $U^*_{p'_0}M$ defined by
$\xi_l<0$ and so be $U_{l-}$ in $U^*_{p_0}M$. 
Taking a sufficiently small neighborhood $W$ of 
$\eta$ in $U_{p_0}^*M$, we define the mapping 
$\psi: \overline{U_{l-}}\cap W\to U'_{l-}$
so that it preserves the values of $H_i$ $(1\le i\le n-1)$, i.e., by 
$\psi(p_0; \xi_1,\dots,\xi_n)
=(p'_0; \tilde\xi_1,\dots,\tilde\xi_n)$, where
\begin{equation*}
\tilde\xi_i=\xi_i\quad (i\ne l),\qquad
\tilde\xi_l=\sqrt{(-1)^{l-1}\prod_{k\ne l}
(f_l(x_l^1)-H_k)}\ .
\end{equation*}
Note that $H_k$'s are functions of
$(p_0; \xi_1,\dots,\xi_n)\in \overline{U_{l-}}$.

Let $\tilde x_l^1$ be the value of $x_l(t,\psi(\eta_s))$ at the time when $\sigma_l(t,\psi(\eta_s))=2(a_{l-1}^--a_l^+)$, which is $-x_l^1$
or $x_l^1+\alpha_l/2$. Also, $\tilde x_l^0$
is similarly defined.
Let $N'$ be the
submanifold of $M$ defined by 
$x_l=\tilde x_l^1$, and define the diffeomorphism
$\Psi: N'\to N_l$ by
\begin{equation*}
\Psi(x_1,\dots,,\tilde x_l^1,\dots,x_n)=
(x_1,\dots,\tilde x_l^0,\dots,x_n).
\end{equation*}
Then we have the following lemma. The proof being similar to that for Lemma \ref{lemma:shift},
we omit.
\begin{lemma}
$\Psi(\gamma(t_2(\psi(\eta_s)),\psi(\eta_s)))=
\gamma(t_2(\eta_s),\eta_s)$ for any $s>0$, 
where $t_2(\eta_s)$
denotes the time when $\sigma_l
(t_2(\eta_s),\eta_s)=2(a_{l-1}^--a_l^+)$.
\end{lemma}
Since $t=t_2(\eta_s)$ is the first positive time 
when the geodesic $\gamma(t,\eta_s)$ reach $N_l$
again, it follows that $t_2(\eta_0)=\lim_{s\to 0}
t_2(\eta_s)$ is the first positive time when
the Jacobi field $Y_l(t)$ vanishes.
Applying Proposition \ref{prop:noconj} to the
geodesic $\gamma(t,\psi(\eta_0))$, we have
$t_0(\psi(\eta_0))<t_2(\psi(\eta_0))$. Since
\begin{equation*}
\sigma_n(t_2(\psi(\eta_s)),\psi(\eta_s))=
\sigma_n(t_2(\eta_s),\eta_s),
\end{equation*}
we then have $\sigma_n(t_2(\eta_0),\eta_0)>2(a_{n-1}^--a_n^+)$, which implies
$t_0(\eta_0)<t_2(\eta_0)$, and hence $Y_l(t_0(\eta_0))\ne 0$.

Next, let us consider the case (i) with the condition $p_0\in J_l$.
Let $\eta_s\in U_{p_0}^*M$ be as above so that
the geodesic $\gamma(t,\eta_0)$ is transversal
to $J_l$ in $N_l$. Then the family of geodesics
$\{\gamma(t,\eta_s)\}_{s>0}$ coincides with the family
$\{\gamma(t,\zeta_r(\eta_{s_0}))\}$ for a fixed $s_0>0$, where $\{\zeta_r\}$ is the one-parameter group
of diffeomorphisms of $U^*M$ generated by
$X_{F_l}$. Thus, in this case, the first positive
time $t_2(\eta_0)$ when the Jacobi field $Y_l(t)$
vanishes has the property that
\begin{equation*}
\gamma(t_2(\eta_0),\eta_s)=\gamma(t_2(\eta_0),
\eta_0)\in J_l\ ,\quad \sigma_l(t_2(\eta_0),\eta_s)=2(a_{l-1}^--a_l^+)\ .
\end{equation*}
Now, let us consider $N_l$ as an $(n-1)$-dimensional Liouville manifold constructed from the constants $a_j$
$(j\ne l)$ and the function $A(\lambda)$. Then the
variables $f_l(x_l)$ and $f_{l+1}(x_{l+1})$ are connected
to a single variable whose range is $[a_{l+1},a_{l-1}]$,
and the total variation of this variable along the
geodesic $\gamma(t,\eta_0)$ $(0\le t\le t_2(\eta_0))$ is
equal to $2(a_{l-1}^--a_{l+1}^+)$. Hence by Proposition
\ref{prop:sigma} for the $(n-1)$-dimensional manifold $N_l$,
we have $t_0(\eta_0)<t_2(\eta_0)$, and thus $Y_l(t_0(\eta_0))
\ne 0$.

Next, we shall consider the case (ii); the geodesic 
$\gamma(t,\eta)$ does not intersects $J_l$. There are
two cases: $a_l=f_l(x_l(t,\eta))=b_{l-1}$; $b_{l+1}=
f_{l+1}(x_{l+1}(t,\eta))=a_l$. The proofs for them are similar, so we may assume $a_l=b_{l-1}$. Note that
$b_l<a_l$ in this case, since $\gamma(t,\eta)$ does not meet
$J_l$.
The Jacobi field $Y_l(t)$ is given by the one-parameter family of geodesics $\{\gamma(t,\eta_s)
\}$, where $\eta_s\in U_{p_0}^*M$ satisfies
$\eta_0=\eta$ and $H_{l-1}(\eta_s)=a_l+s^2$, 
$H_j(\eta_s)=b_j$ for $j\ne l-1$. Define $\theta_s(t)$ by the formula
\begin{equation*}
f_l(x_l(t,\eta_s))=a_l(\cos\theta_s(t))^2+H_{l-1}(\eta_s)
(\sin\theta_s(t))^2\ ,\quad \theta_s(0)=0
\end{equation*}
and put $\theta_0(t)=\lim_{s\to 0}\theta_s(t)$.
Let $t_2(\eta)$ be the time such that 
$\theta_0(t_2(\eta))=\pi$. Then $t=t_2(\eta)$ is the first
positive time when $Y_l(t)=0$. We shall show that
$t_0(\eta)<t_2(\eta)$.
We have
\begin{gather*}
\sum_{i\ne l}\int_0^{\sigma_i(t_2(\eta),\eta)}\frac{(-1)^iG_{l,l-1}(f_i)A(f_i)\ d\sigma_i}
{|f_i-a_l|\sqrt{-\prod_{k\ne l-1}(f_i-b_k)
\cdot\prod_{k\ne l}(f_i-a_k)}}\\
+\frac{(-1)^l2\pi\ G_{l,l-1}(a_l)A(a_l)}
{\sqrt{-\prod_{k\ne l-1}(a_l-b_k)
\cdot\prod_{k\ne l}(a_l-a_k)}}
=0\ .
\end{gather*}
Also, a similar observation as in the proof of Lemma \ref{lemma:flat} indicates
\begin{gather*}
-2\sum_{i\ne l}\int_{a_i^+}^{a_{i-1}^-}
\frac{(-1)^{i}G_{l,l-1}(\lambda)A(a_l)\ d\lambda}
{|\lambda-a_l|\sqrt{-\prod_{k\ne l-1}(\lambda-b_k)
\cdot\prod_{k\ne l}(\lambda-a_k)}}\\
=\frac{(-1)^l2\pi\ G(a_l)A(a_l)}
{\sqrt{-\prod_{k\ne l-1}(a_l-b_k)
\cdot\prod_{k\ne l}(a_l-a_k)}}\ .
\end{gather*}
Thus we have the formula:
\begin{equation}\label{eq:theta7}
\begin{gathered}
2\sum_{i\ne l}\int_{a_i^+}^{a_{i-1}^-}
\frac{A(\lambda)-A(a_l)}{|\lambda-a_l|}
\frac{(-1)^{i}G_{l,l-1}(\lambda)\ d\lambda}
{\sqrt{-\prod_{k\ne l-1}(\lambda-b_k)
\cdot\prod_{k\ne l}(\lambda-a_k)}}\\
=\sum_{i\ne l}\int_{\sigma_i(t_2(\eta),\eta)}^{2(a_{i-1}^--a_i^+)}\frac{(-1)^iG_{l,l-1}(f_i)A(f_i)\ d\sigma_i}
{|f_i-a_l|\sqrt{-\prod_{k\ne l-1}(f_i-b_k)
\cdot\prod_{k\ne l}(f_i-a_k)}}\ .
\end{gathered}
\end{equation}
Take a sufficiently large constant $c>0$ and put
\begin{equation*}
B(\lambda)=c-\frac{A(\lambda)-A(a_l)}{\lambda-a_l}\ ,
\qquad [i]=i\ (i<l);\quad =i-1\ (i>l)\ .
\end{equation*}
Then, by Lemma \ref{lemma:flat} ($(n-1)$-dimensional
case), the left-hand side of the formula \eqref{eq:theta7}
is rewritten as
\begin{equation*}
2\sum_{[i]=1}^{n-1}\int_{a_i^+}^{a_{i-1}^-}
\frac{(-1)^{[i]+1}G_{l,l-1}(\lambda)B(\lambda)\ d\lambda}
{\sqrt{-\prod_{[k]=1}^{n-1}((\lambda-a_{k}^-)
(\lambda-a_{k}^+))}}\ .
\end{equation*}
Since $B(\lambda)$ satisfies the condition (\ref{cond2}), the
above value is positive by Proposition \ref{prop:cond2} (1) 
($(n-1)$-dimensional case). If $t_0(\eta)\ge t_2(\eta)$, 
then, applying 
Proposition \ref{prop:sigma} to the Liouville manifold
$N_l$, we have $\sigma_i(t_2(\eta),\eta)\le 2(a_{i-1}^--a_i^+)$
for any $i\ne l$. This indicates that the right-hand side 
of the
formula \eqref{eq:theta7} is nonpositive, a contradiction.
Therefore, it 
follows that $t_0(\eta)<t_2(\eta)$, and $Y_l(t_0(\eta))\ne 
0$.

Next, we shall consider the case (iii); $\gamma(t,\eta)$ is 
tangent to $J_l$, but not contained in it. First, we assume
$p_0\not\in J_l$. In this case, it holds that either 
$f_{l+1}(x_{l+1}^0)
<b_l=a_l=f_l(x_l^0)=b_{l-1}$ or $b_{l+1}=f_{l+1}(x_{l+1}^0)
=a_l=b_l<f_l(x_l^0)$.
Since the proofs are similar, we may assume 
\begin{equation*}
f_{l+1}(x_{l+1}^0)
<b_l=a_l=f_l(x_l^0)=b_{l-1}\ .
\end{equation*}
Define a one-parameter family of unit covectors $\eta_s$
at $p_0$ such that $\eta_0=\eta$, $H_{l}(\eta_s)=a_l-s^2$,
and $H_j(\eta_s)=b_j$ for $j\ne l$. Then, the geodesics
$\gamma(t,\eta_s)$ $(s\ne 0)$ are still on $N_l$, but do 
not meet $J_l$. Since the zeros of a 
family of Jacobi fields are continuously depending on the
parameter, it follows that $\lim_{s\to 0}t_2(\eta_s)=
t_2(\eta)$ represents the first positive time $t$ such that
$Y_l(t)=0$.
Now, substitute $\eta=\eta_s$ in the formula \eqref{eq:theta7} and take a limit $s\to 0$. Then, if
$t_0(\eta)\ge t_2(\eta)$, one gets a similar contradiction
as above. Thus we have $t_0(\eta)<t_2(\eta)$, and $Y_l(t_0(\eta))\ne 0$ in this case.

Next, we assume that $p_0\in J_l$. Let $\eta_s\in U_{p_0}^*M$ 
$(\eta_0=\eta)$ be a one-parameter family of covectors such that the infinitesimal variation of the geodesics $\{\gamma(t,\eta_s)\}$ at $s=0$ is equal to $Y_l(t)$. Let $t_2(\eta_s)$ be the
first positive time such that $\gamma(t,\eta_s)\in N_l$.
Then, $t_2(\eta)=\lim_{s\to 0}t_2(\eta_s)$ is the first
positive time such that $Y_l(t)=0$. Also, by the same reason
as in the case (i), we have $\gamma(t_2(\eta_s),\eta_s)\in 
J_l$ and so does for $s=0$. Hence we have $\sigma_{l+1}(t_2(
\eta),\eta)=2(a_l^--a_{l+1}^+)$, and thus $t_0(\eta)<t_2(
\eta)$ by Proposition \ref{prop:sigma}. 

Finally, let us consider the case (iv); $\gamma(t,\eta)$ is 
contained in $J_l$. In this case, we have
\begin{equation*}
b_{l+1}=f_{l+1}(x_{l+1}^0)
=b_l=a_l=f_l(x_l^0)=b_{l-1}\ .
\end{equation*}
Define the one-parameter family of the initial points $p_0(s)$ and the initial covectors $\eta_s\in U_{p_0(s)}^*M$
so that $H_{l+1}(\eta_s)=H_l(\eta_s)=b_l-s^2$ and $H_i(\eta_s)= b_i$ $(i\ne l,l+1)$. Then the formula \eqref{eq:theta7} is valid for $\eta_s$. Taking a limit $s\to 0$, we have:
\begin{gather*}
2\sum_{\substack{1\le [i]\le n-1\\ [i]\ne l}}
\int_{a_i^+}^{a_{i-1}^-}
\frac{(-1)^{[i]+1}G_{l,l-1}(\lambda)B(\lambda)\ d\lambda}
{\sqrt{-\prod_{[k]=1}^{n-1}((\lambda-a_{k}^-)
(\lambda-a_{k}^+))}}\\
=\sum_{i\ne l,l+1}\int_{\sigma_i(t_2(\eta),\eta)}^{2(a_{i-1}^--a_i^+)}\frac{(-1)^iG_{l,l-1}(f_i)A(f_i)\ d\sigma_i}
{|f_i-a_l|\sqrt{-\prod_{k\ne l-1}(f_i-b_k)
\cdot\prod_{k\ne l}(f_i-a_k)}}\ .
\end{gather*}
Since the left-hand side of the above formula is positive by
Proposition \ref{prop:cond2},
we have $t_0(\eta)<t_2(\eta)$ as before. This completes the
proof of Proposition \ref{prop:noconj2}.
\section{Cut locus (3)}
In this section, we shall give a proof of Theorem 
\ref{thm:cut} (2) for the cases (III) and (IV), and (3) for 
the case (V). First, we shall consider the case (III); $p_0
\in N_n$. 

We use Lemma \ref{lemma:shift} in the case where $x_n^1=0$ 
and use it by exchanging $p_0$ and $p'_0$.
As a consequence, we see that the mapping
\begin{equation*}
(U_{p_0}^*M\supset)\ U_+\ni \eta\longmapsto\gamma(t_0(\eta),\eta)\in N_n
\end{equation*}
is a $C^\infty$ embedding. Therefore, to prove (2) in this
case it is enough to show that the mapping
\begin{equation}\label{eq:partial}
\partial\overline{U_+}\ni \eta\longmapsto\gamma(t_0(\eta),\eta)\in N_n
\end{equation}
is an embedding.

For $p_0\in N_{n}$ and $\eta\in U^*_{p_0}N_n$, let $\tilde
t_0(\eta)$ denotes the value which is defined in the same way 
as
$t_0(\eta)$ for the $(n-1)$-dimensional Liouville manifold 
$N_n$. (Note that $N_n$ is constructed from the constants 
$0<a_{n-1}<\dots<a_0$ and the function $A(\lambda)$ as in
\S2.)  As we have proved in (1), $t=\tilde t_0(\eta)$ gives the
cut point of $p_0$ along the geodesic $\gamma(t,\eta)$
in $N_n$. In particular, we have $t_0(\eta)\le
\tilde t_0(\eta)$. Therefore, the following proposition will indicate
that the mapping \eqref{eq:partial} is an embedding.

\begin{prop}\label{prop:sect7}
$t_0(\eta)<\tilde t_0(\eta)$ for any $p_0\in N_{n}$ and 
$\eta\in U^*_{p_0}N_n$.
\end{prop}
\begin{proof}
We use the formula
\begin{gather*}
\sum_{i=1}^{n-1}\int_{a_i^+}^{a_{i-1}^-}\frac{
(-1)^{i+1}G_{n-1,n-2}(\lambda)B(\lambda)}{
\sqrt{-\prod_{k=1}^{n-2}(\lambda-b_k)\prod_{k=0}^{n-1}
(\lambda-a_k)}}\,d\lambda\\
=\sum_{i=1}^{n-1}\int_{\sigma_i(t_0(\eta),\eta)}^{
2(a_{i-1}^--a_i^+)}\frac{
(-1)^{i}G_{n-1,n-2}(f_i) A(f_i)}{
(f_i-a_n)
\sqrt{-\prod_{k=1}^{n-2}(f_i-b_k)\prod_{k=0}^{n-1}
(f_i-a_k)}}\,d\sigma_i\ ,
\end{gather*}
where
\begin{equation*}
B(\lambda)=c-\frac{A(\lambda)-A(a_n)}{\lambda-a_n}
\end{equation*}
and $c>0$ is a sufficiently large constant. As 
before, the left-hand side of the above formula is 
positive, whereas each integrand of the right-hand 
side is negative for $i\le n-2$. Thus, if
$t_0(\eta)=\tilde t_0(\eta)$, then
\begin{equation*}
2(a_{n-2}^--a_{n-1}^+)=\sigma_{n-1}(\tilde t_0(\eta),
\eta)=\sigma_{n-1}(t_0(\eta),\eta),
\end{equation*}
and we have a contradiction. Therefore it follows 
that $t_0(\eta)<\tilde t_0(\eta)$.
\end{proof}
Next, we shall consider the case (IV); $x_n^0=\alpha_n/4$
and $p_0\not\in J_{n-1}$. By the similar fact  as 
Lemma \ref{lemma:shift} and by
the proved cases, we see that the map $\eta\mapsto
\gamma(t_0(\eta),\eta)$ gives $C^\infty$ embeddings
$U_+\to N$ and $\partial \overline{U_+}\to N$,
where $N$ is the subset of $N_{n-1}$ such that
$x_n=-\alpha_n/4$. To see that
the cut locus $C(p_0)$, the union of the images of 
those maps, is in the interior of $N$, it is enough to show
that $C(p_0)$ does not meet $J_{n-1}$, a connected component
of which is equal to the boundary of $N$. Assume that
$\gamma(t_0(\eta),\eta)\in J_{n-1}$ for some $\eta\in
\overline{U_+}$. By Lemma \ref{lemma:subm} we see that
$F_{n-1}(\eta)=0$. Since $p_0\not\in J_{n-1}$ and $p_0\in
N_{n-1}$, it thus follows that $\eta\in U^*_{p_0}N_{n-1}$,
i.e., $\eta\in\partial\overline{U_+}$. Now put
\begin{equation*}
\gamma(t)=\gamma(t_0(\eta)-t,\eta)
\end{equation*}
Then, $\gamma(t)$ is a geodesic starting at $\gamma(t_0(\eta),
\eta)\in J_{n-1}$ and its first conjugate point is $p_0=
\gamma(t_0(\eta))$. But, as we shall see just below, the first
conjugate point of any geodesic starting at a point in 
$J_{n-1}$ also belongs to $J_{n-1}$, which is a contradiction.
Thus $C(p_0)$ is contained in the interior of $N$.
This finishes the proof of (2) of the theorem
in this case.

Finally we prove the statement (3) of the theorem
for the case (V); $p_0\in J_{n-1}$. Note that
$t=t_0(\eta)$
gives the cut point of $p_0$ along the geodesic 
$\gamma(t,\eta)$ for any $\eta\in U^*_{p_0}M$.
We apply the results proved above
to the $(n-1)$-dimensional Liouville manifold $N_{n-1}$,
which is constructed from the constants $0<a_n<a_{n-2}<\dots
<a_0$ and the function $A(\lambda)$. 
Noting the fact $J_{n-1}\cap J_{n-2}=\emptyset$,
we see that the cut locus $\tilde C(p_0)$
of $p_0$ in $N_{n-1}$ is an $(n-2)$-closed disk, and it is the
image of the map
\begin{equation*}
\overline{U_+}\cap T^*_{p_0}N_{n-1}\to J_{n-1},
\qquad \eta\mapsto \gamma(\bar t_0(\eta),\eta),
\end{equation*}
where $\bar t_0(\eta)$ is the value which is defined in the 
same way as
$t_0(\eta)$ for the $(n-1)$-dimensional Liouville manifold 
$N_{n-1}$. It has also been proved that the above map is
an embedding on the interior and on the boundary.

Let $\tilde \eta$ be a unit covector such that
$\tilde \eta\not\in T^*_{p_0}N_{n-1}$. Let $\{\zeta_s\}$
be the one-parameter transformation group of $T^*M$ generated 
by $X_{F_{n-1}}$. Then $\tilde\eta_s=\zeta_s(\tilde \eta)\in U^*_{p_0}M$
whose orthogonal projection to $T^*_{p_0}J_{n-1}$ 
does not depend on $s$,
and $\tilde\eta_{\pm\infty}=\lim_{s\to\pm\infty}\tilde\eta_s\in
T^*_{p_0}N_{n-1}$.  By the definition of $t_0(\tilde\eta_s)$
we have $\gamma(t_0(\tilde \eta_s),\tilde\eta_s)\in J_{n-1}$.
Therefore the Jacobi field $\pi_*X_{F_{n-1}}$ along the
geodesic $\gamma(t,\tilde\eta_s)$ also vanish at $t=
t_0(\tilde\eta_s)$. Thus we have
\begin{equation*}
\gamma(t_0(\tilde \eta_s),\tilde\eta_s)=
\gamma(t_0(\tilde\eta_{\pm\infty}),\tilde\eta_{\pm\infty}),\qquad
t_0(\tilde\eta_s)=t_0(\tilde\eta_{\pm\infty})
\end{equation*}
for any $s\in\R$. Since $t=t_0(\tilde\eta_s)$
gives the cut point of $p_0$ along the geodesic
$\gamma(t,\tilde\eta_s)$, and since $\tilde\eta_{+\infty}
\in U^*N_{n-1}$ and $\tilde\eta_{-\infty}\in U^*N_{n-1}$
are symmetric with respect to the hyperplane 
$T^*_{p_0}J_{n-1}\subset T^*_{p_0}N_{n-1}$, it follows that
$\bar t_0(\eta_{\pm\infty})=t_0(\eta_{\pm\infty})$.
Thus we have proved that the cut locus $C(p_0)$ of $p_0$ in
$M$ coincides with $\tilde C(p_0)$ and that if $\eta_1$,
$\eta_2\in U^*_{p_0}M$
have the same $T^*_{p_0}J_{n-1}$-components, then $\gamma(
t_0(\eta_1),\eta_1)=\gamma(t_0(\eta_2),\eta_2)$. From these
it also follows that for $\eta\in U^*_{p_0}J_{n-1}$, 
$t=t_0(\eta)$ gives the first conjugate
point of $p_0$ with multiplicity two along the geodesic 
$\gamma(t,\eta)$. This finishes the proof of Theorem
\ref{thm:cut}.


\end{document}